\documentclass[11pt]{amsart}  

\usepackage[latin1]{inputenc}
\usepackage{amsmath}
\usepackage{amsfonts}
\usepackage{amssymb}
\usepackage{stmaryrd}
\usepackage{latexsym}
\usepackage{graphicx}
\usepackage{subfigure}
\usepackage{color}
\usepackage{mathtools}
\usepackage{verbatim}
\usepackage[all]{xy}
\usepackage{graphics}
\usepackage{pdfsync}
\usepackage[backref=page]{hyperref}
\usepackage{ytableau}
\usepackage{xcolor}
\usepackage{tikz}
\usepackage[normalem]{ulem}
\usepackage{enumitem}

\theoremstyle{plain}
\newtheorem{theorem}{Theorem}[section]
\newtheorem{proposition}[theorem]{Proposition}

\newtheorem{lemma}[theorem]{Lemma}
\newtheorem{corollary}[theorem]{Corollary}

\theoremstyle{definition}
\newtheorem{definition}[theorem]{Definition}
\newtheorem{example}[theorem]{Example}

\theoremstyle{remark}
\newtheorem{remark}[theorem]{Remark}

\numberwithin{equation}{section}
\newcommand{\bN}{\mathbb{N}}

\newcommand{\bP}{\mathbb{Z}_+}
\newcommand{\bQ}{\mathbb{Q}}

\newcommand{\bZ}{\mathbb{Z}}

\newcommand{\Sym}{\mathrm{Sym}}

\newcommand{\suchthat}{\;|\;}

\newcommand{\alpx}{\mathbf{x}}

\newcommand{\cat}{\mathrm{Cat}}
\newcommand{\ds}[2]{\big\langle {#1}\big\rangle_{#2}}

\newcommand{\schub}[1]{\mathfrak{S}_{#1}}

\newcommand{\qsym}{\mathrm{QSym}} 

\newcommand{\supp}{\mathrm{Supp}}
\newcommand{\bfc}{{\bf c}}
\newcommand{\alpxn}[1]{\mathbf{x}_{#1}} 
\newcommand{\alpxplus}{\mathbf{x}_+}

\newcommand{\psymb}{\mathsf{P}} 
\newcommand{\qsymb}{\mathsf{Q}} 

\newcommand{\slide}{\mathfrak{F}} 
\newcommand{\fund}{L} 
\newcommand{\backforest}{\overleftarrow{\mathfrak{P}}} 
\newcommand{\forestpoly}{\mathfrak{P}} 

\newcommand{\rootlist}{\mathsf{rl}}
\newcommand{\IF}{\mathsf{IndFor}}
\newcommand{\lbs}{\mathsf{LBS}}

\newcommand{\can}{\mathrm{can}}
\newcommand{\interval}{\mathsf{INT}}

\newcommand{\wf}{\Omega}
\newcommand{\wfcorr}{{\wf^\bullet}}
\newcommand{\wfequiv}{\underset{\wf}{\equiv}}
\newcommand{\simwf}{\underset{\wf}{\sim}}

\newcommand{\Zaug}{\overline{\bZ}}

\newcommand{\lf}[2]{#1^{[#2]}}         
\newcommand{\Zleaves}{\underline{\bZ}}

\newcommand{\injwords}[1]{\mathrm{Inj}(#1)}
\newcommand{\flag}{\operatorname{Fl}} 
\newcommand{\quasicoin}{\operatorname{SC}} 
\newcommand{\perm}{\operatorname{Perm}}
\newcommand{\red}{\operatorname{Red}}
\newcommand{\nvects}{\mathsf{Codes}} 
\newcommand{\nvectsp}{\nvects_+} 
\newcommand{\nvectsn}[1]{\nvects_{#1}} 
\newcommand{\artin}[1]{\mathsf{Artin}_{#1}} 
\newcommand{\abb}[1]{\mathsf{ABB}_{#1}} 
\newcommand{\abbtop}[1]{\abb{#1}^{\mathrm{top}}}
\newcommand{\lsupp}{\mathrm{LSupp}}
\newcommand{\flatten}{\mathrm{flat}}
\newcommand{\internal}{\operatorname{IN}}
\newcommand{\terminal}{\operatorname{TN}}
\newcommand{\backquasi}{\overleftarrow{QR}} 
\newcommand{\compatible}{\mathsf{Compatible}}
\newcommand{\lcode}{\mathsf{lcode}}
\newcommand{\decreasing}{\mathsf{Dec}} 
\newcommand{\compl}{\mathsf{LC}} 
\newcommand{\wt}{\mathrm{wt}} 
\newcommand{\trim}[1]{\mathrm{trim}_{#1}} 
\newcommand{\injective}{\mathrm{Inj}} 

\title{Forest polynomials and the class of the permutahedral variety 
}

\author{Philippe Nadeau}
\address{Univ Lyon, Universit\'e Claude Bernard Lyon 1, CNRS UMR
5208, Institut Camille Jordan, 43 blvd. du 11 novembre 1918, F-69622 Villeurbanne cedex, France}
\email{\href{mailto:nadeau@math.univ-lyon1.fr}{nadeau@math.univ-lyon1.fr}}

\author{Vasu Tewari}
\address{Department of Mathematics, University of Hawaii at Manoa, Honolulu, HI 96822, USA}
\email{\href{mailto:vvtewari@math.hawaii.edu}{vvtewari@math.hawaii.edu}}

\thanks{P.~N is partially supported by the project ANR19-CE48-011-01 (COMBIN\'E). V.~T. acknowledges the support from Simons Collaboration Grant \#855592.}

\begin{document}

\begin{abstract}
We study a basis of the polynomial ring that we call forest polynomials.
This family of polynomials is indexed by a combinatorial structure called indexed forests and permits several definitions, one of which involves flagged P-partitions.
 As such, these polynomials have a positive expansion in the basis of slide polynomials. 
 By a novel insertion procedure that may be viewed as a generalization of the Sylvester correspondence we establish that Schubert polynomials decompose positively in terms of forest polynomials. 
 Our insertion procedure involves a correspondence on words which allows us to show that forest polynomials multiply positively.

We proceed to show that forest polynomials are a particularly convenient basis in regards to studying the quotient of the polynomial ring modulo the ideal of positive degree quasisymmetric polynomials.
 This aspect allows us to give a manifestly nonnegative integral description for the Schubert class expansion of the cohomology class of the permutahedral variety in terms of a parking procedure. 
 We study the associated combinatorics in depth and introduce a multivariate extension of mixed Eulerian numbers.
\end{abstract}

\maketitle
\tableofcontents

\section{Introduction}
\label{sec:intro}

\subsection{Background}
\label{subsec:bg}

The coinvariant algebra $R_n$ in type $A$, which is the quotient of the polynomial ring $\bQ[\alpxn{n}]\coloneqq \bQ[x_1,\dots,x_n]$ modulo the ideal $\Sym_n^+$ of positive degree symmetric polynomials, has a  storied past and a vibrant present. By work of Borel \cite{Bor53}, $R_n$ is isomorphic to the cohomology ring $H^*(\flag_n,\bQ)$ of the complete flag variety $\flag_{n}$. As a $\bQ$-vector space, $R_n$ is $n!$-dimensional and a distinguished basis of monomials was provided by Artin \cite{EArt44}.
As an $S_n$-representation, $R_n$ carries a graded regular representation as was shown by Chevalley \cite{Che55}.
In recent years, several interesting generalizations of $R_n$ have been studied to great effect, in particular in the theory of Macdonald polynomials \cite{GLW21,HRS18,Rho22}.

A basis for $H^*(\flag_n,\bQ)$ is given by Schubert classes $\{[X_w]\}_{w\in S_n}$ and, following on work of Bernstein--Gelfand--Gelfand \cite{BGG73}, a distinguished family $\{\schub{w}\}_{w\in S_n}$ of polynomial representatives  for these classes  was introduced by Lascoux--Sch\"utzenberger \cite{Las82}.
These representatives, known as Schubert polynomials, have proven to be influential in combinatorial algebraic geometry; see \cite{Ber93,BS98,Bil93,FK96,FS94,HPW22,KM05,Lam18} to get a measure of their importance, with the caveat that the preceding list is by no means exhaustive.
Schubert polynomials are also one of the main protagonists of this article, and we proceed to describe another.

A ring that contains symmetric functions, and whose combinatorics simultaneously draws upon and in return enriches the latter, is the ring of quasisymmetric functions introduced by Gessel \cite{Ges84}.
Truncating a quasisymmetric function, which is a bounded degree formal power series in $\bQ[\![x_1,x_2,\dots]\!]$, to finitely many variables $x_1$ through $x_n$ produces a quasisymmetric polynomial.
Let $\qsym_n^+$ denote the ideal in $\bQ[\alpxn{n}]$ generated by positive degree quasisymmetric polynomials in $x_1$ through $x_n$.
Motivated by the rich vein of ideas in the study of $R_n$, Aval--Bergeron--Bergeron \cite{ABB04} considered its `quasisymmetric analogue', i.e. the quotient $\quasicoin_n\coloneqq \bQ[\alpxn{n}]/\qsym_n^+$.
They proved the foundational result that as a $\bQ$-vector space, $\quasicoin_n$ has dimension given by the $n$th Catalan number $\mathrm{Cat}_{n}\coloneqq \frac{1}{n+1}\binom{2n}{n}$.
In particular, they produced a standard basis of monomials-- henceforth referred to as \emph{ABB monomials}-- indexed by Dyck paths for $\quasicoin_n$.
Nevertheless, numerous questions remain-- Is there a geometric interpretation akin to Borel's realization of $R_n$?
Is there a connection to the geometry of $\flag_n$?
Is there a basis of the polynomial ring that mirrors the pleasant aspects of Schubert polynomials and behaves `nicely' with regards to reduction modulo $\qsym_n^+$?

\smallskip

In this article we address some of these questions.
Our motivation is geometric and, at first glance, has nothing to do with quasisymmetric polynomials.
The permutahedral variety $\perm_n$ is a smooth projective toric variety of dimension $n-1$ determined by the normal fan of the usual permutahedron of dimension $n-1$.
As it is a subvariety of $\flag_n$, its cohomology class $[\perm_n]\in H^*(\flag_n,\bQ)$ expands in terms of Schubert classes $[X_w]$ as follows:
\begin{align}
  [\perm_n]=\sum_{w\in S_n}a_w[X_w],
\end{align}
where $S_n$ denotes the symmetric group on $n$ letters. In fact one can restrict the sum to range over permutations with length $n-1$.
The coefficients $a_w$ are naturally in $\bZ_{\geq 0}$.
We find ourselves faced with a quintessential combinatorial task: assign $a_w$ a combinatorial meaning.

In \cite{NT20}, the authors gave an explicit formula for the $a_w$ by relying on work of Klyachko describing a Macdonald-like formula for the image of a Schubert class in $H^*(\perm_n,\bQ)^{S_{n}}$.
The resulting expression involves a sum of \emph{normalized mixed Eulerian numbers}.
So while we are able to infer positivity of $a_w$ via an explicit formula, the expression itself involves summing rational numbers.
This aspect motivated a deeper study of mixed Eulerian numbers; see  \cite{NT21,NTremixed} where a $q$-analogue is introduced.
There has been other recent work utilizing mixed Eulerians in a crucial way; see \cite{Ber20, Hor21}. For a matroidal generalization, see \cite{KaKu23}.

The quotient $\quasicoin_n$ enters the setup as follows.
\emph{Divided symmetrization} (DS henceforth), introduced by Postnikov \cite[Section 3]{Pos09}, takes a homogeneous polynomial $f(x_1,\dots,x_n)\in \bQ[\alpxn{n}]$ of degree $n-1$ as input and outputs $\ds{f}{n}\in \bQ$
as:
\begin{align}
\label{eq:def_ds}
  \ds{f}{n}\coloneqq \sum_{w\in S_n}w\cdot\left(\frac{f(x_1,\dots,x_n)}{\prod_{1\leq i\leq n-1}(x_i-x_{i+1})}\right).
\end{align}
By work of Anderson--Tymoczko \cite{And10}, the $a_w$ permit the following description in terms of DS:
\begin{align}
  \label{eq:aw_via_ds}
a_w=\ds{\schub{w}}{n}.
\end{align}
In \cite{DS} the authors showed that $\ds{f}{n}$ may be computed by reduction modulo $\qsym_n^+$, thereby setting the stage for the current work. More precisely, given homogeneous $f\in \bQ[\alpxn{n}]$ of degree $n-1$, let $\overline{f}$ denote the representative in $\quasicoin_n$ expressed in the basis of ABB monomials.
Then \cite[Theorem 1.3]{DS} states that
\begin{align}
  \label{eq:ds_and_qsym}
  \ds{f}{n}=\overline{f}(1,\dots,1).
\end{align}
Thus Equations~\eqref{eq:ds_and_qsym} and ~\eqref{eq:aw_via_ds} together motivate the investigation the reduction of Schubert polynomials modulo $\qsym_n^+$.
As we shall show, $\schub{w}\mod \qsym_n^+$ expands positively in terms of ABB monomials, thereby allowing us to give a manifestly nonnegative integral description for $a_w$.

\subsection{Statement of results}
\label{subsec:results}

We begin by describing a combinatorial model for $a_w$, and thereafter proceeding to discuss the moving parts.

Recall that $i_1i_2\ldots i_k$ is a \emph{reduced word} for a permutation $w$ if $w=s_{i_1}s_{i_2}\ldots s_{i_k}$ where $s_i$ is the transposition $(i\leftrightarrow i+1)$, and $k$ is minimal, given by the length $\ell(w)$.
We denote the set of reduced words for $w$ by $\red(w)$.
With this data at hand, we describe a new parking procedure that underlies the combinatorics for $a_w$.

\smallskip

\noindent {\em Parking procedure $\wf$}:  Consider parking spots indexed by $\bZ$, initially all empty.
Cars $1,2,\dots$ arrive successively, with car $i$ preferring spot $v_i$, and want to park at (empty) spots.
Assume inductively that $i-1$ cars have already parked.
If spot $v_i$ is empty, then car $i$ parks there.
Otherwise, $v_i$ belongs to an interval $[a,b]\coloneqq \{a,a+1,\dots,b-1,b\}$ of occupied spots with spots $a-1$ and $b+1$ being free.
Define $v_j$ to be the preferred spot of the car that parked last in  $[a,b]$; that is, $j<i$ is maximal such that $v_j\in [a,b]$. Then, car $i$ parks in $b+1$ if $v_i\geq v_j$, while  it parks in $a-1$ if $v_i<v_j$.

After $k$ cars have parked they occupy a $k$-subset $\wf(v_1\cdots v_k)\subset \bZ$. A preference word $v_1\cdots v_k$ is called a \emph{$\wf$-parking word} if $\wf(v_1\cdots v_k)=\{1,\ldots,k\}$.

\begin{theorem}
\label{th:aw_as_parking}
If $w\in S_n$ has length $n-1$, then $a_w$ is the number of reduced words of $w^{-1}$ that are also $\wf$-parking words.
\end{theorem}

\begin{example}
\label{ex:running}
Consider $w=21543\in S_5$ with $\ell(w)=4$. To compute $a_w$, we need to compute $\wf$-parking words in $\red(w^{-1})=\red(w)$. Here are all reduced words for $21543$.
\begin{align*}
\red(w)=\{1343,3143,3413,3431,1434,4134,4314,4341\}.
\end{align*}
 The first four  are $\wf$-parking words, and therefore $a_{21543}=4$.
\end{example}

While the statement of Theorem~\ref{th:aw_as_parking} is succinct, we need some new theory to arrive at it. 
We introduce some notions that are necessary and refer the reader to Section~\ref{sec:bg} for more.

Let $\nvectsp$ denote the set of sequences $\bfc=(c_1,c_2,\dots)$ of nonnegative integers where only finitely many entries are nonzero.
Our main combinatorial gadget in this article is a family of polynomials $\{\forestpoly_{\bfc}\}_{\bfc\in \nvectsp}$ that we call \emph{forest polynomials}.
They were introduced by the authors in \cite{NT23} using a flagged analogue of $P$-partitions.
It transpires that they possess interesting properties which make them worthy of investigation independently.
The next result captures some aspects of forest polynomials, with an emphasis on how they interact with known bases of the polynomial ring $\bQ[x_1,x_2,\dots]$ such as Schubert polynomials and slide polynomials. For convenience, given two bases $\mathcal{A}$ and $\mathcal{B}$ of $\bQ[x_1,x_2,\dots]$, we say that $\mathcal{A}$ expands \emph{positively} in $\mathcal{B}$ if every element in $\mathcal{A}$ expands in terms of elements in $\mathcal{B}$ with nonnegative coefficients.
\begin{theorem}
\label{th:main_2}
The set of forest polynomials forms a basis for the polynomial ring $\bQ[x_1,x_2,\dots]$. Furthermore, they expand positively in slide polynomials.
\end{theorem}

One upshot of Theorem~\ref{th:main_2} is the potential parallel with Schubert polynomials that it hints at.
Indeed Schubert polynomials also give a basis for $\bQ[x_1,x_2,\dots]$ and, in addition, expand positively in slide polynomials \cite{AssSea17,Bil93}.
In fact, as we demonstrate, the parallel is deeper, and so is the connection-- Schubert polynomials expand positively in forest polynomials.

Toward establishing this, in view of Theorem~\ref{th:main_2} one could hope coarsen the slide expansion of Schubert polynomials so that the summands give forest polynomials.
Since this expansion is obtained by writing down a slide polynomial\footnote{Technically, a slide polynomial or $0$. Working in the back stable setting fixes this anomaly.} for each element of $\red(w)$, we are looking to group elements of $\red(w)$ according to a hitherto unknown rule.
We are thus led to the correspondence $\wfcorr$, of which the $\wf$-parking procedure is but a shadow.
Briefly put, a simplistic version of this correspondence associates with any word $w$ in the alphabet $\bZ_{\geq 1}$ an ordered pair of labeled forests $(\psymb(w),\qsymb(w))$ by an insertion procedure.
Sticking with tradition, we refer to $\psymb(w)$ (respectively $\qsymb(w)$) as the \emph{P-symbol} (respectively \emph{Q-symbol}).
The insertion procedure is such that $\psymb(w)$ is a local binary search forest whereas $\qsymb(w)$ is a decreasing forest; see Section~\ref{sec:wf} for  definitions.

The $\wfcorr$-correspondence allows us to define the equivalence $\wfequiv$ on words; two words are equivalent if their P-symbols coincide.
It transpires that for any permutation $w$, the set $\red(w)$ is closed under $\wfequiv$, and therefore decomposes into $\wfequiv$-equivalence classes.
To each such class $\mathcal{C}$  is attached a polynomial $\forestpoly(\mathcal{C})$ that coincides with a forest polynomial. 
We then have the following as one of our main results.

\begin{theorem}\label{th:main_3}
  Given a permutation $w$, let $\mathcal{G}_w$ denote the set of $\wfequiv$-equivalence classes of $\red(w^{-1})$.
 The Schubert polynomial $\schub{w}$ expands positively in forest polynomials as follows:
  \begin{align*}
    \schub{w}=\sum_{\mathcal{C}\in\mathcal{G}_w}\forestpoly(\mathcal{C}).
  \end{align*}
\end{theorem}

Theorem~\ref{th:main_3} paves the way for understanding reduction of Schubert polynomials modulo $\qsym_n^+$ by translating the question to reducing forest polynomials.
Fortunately, the latter task turns out to be  easy.
To emphasize the parallel to Schubert polynomials, we recall another property thereof.

Fix a positive integer $n$. Let $\nvectsn{n}$ denote the subset of $\nvectsp$ comprising $\bfc=(c_1,c_2,\dots)$ such that $c_i=0$ for all $i>n$.
Indexing Schubert polynomials by (Lehmer) codes of permutations in $S_{\infty}$, we have that the set $\{\schub{\bfc}\suchthat \bfc \in \nvectsn{n}\}$ is a basis of $\bQ[\alpxn{n}]$.
In fact, this basis is nice in regards to reduction modulo $\Sym_n^+$. Letting $\alpx^{\bfc}\coloneqq \prod_{i}x_i^{c_i}$ we have the following decomposition \cite{EArt44}:
\begin{align}
\label{eq:artin}
\bQ[\alpxn{n}]=\bQ\{\alpx^{\bfc}\suchthat \bfc\in \nvectsn{n} \text{ satisfying } c_i\leq n-i \text{ for } i=1,\dots,n\}\oplus \Sym_n^+.
\end{align}
The monomials in the first summand in~\eqref{eq:artin} are called \emph{Artin} monomials.
Let $\artin{n}$ denote the set of $\bN$-vectors that determine Artin monomials by way of their exponent vectors.
In view of~\eqref{eq:artin}, one may ask for representatives in $\bQ[\alpxn{n}]/\Sym_n^+$ in terms of Artin monomials.
We then have the following well-known result describing how Schubert polynomials $\schub{\bfc}$ for $\bfc\in \nvectsn{n}$ perform.
\begin{align}
  \label{eq:schubert_property}
  \schub{\bfc}\mod \Sym_n^+=\left\lbrace\begin{array}{ll}
  \schub{\bfc}  & \bfc\in \artin{n}\\
  0 & \bfc\notin \artin{n}. \end{array}\right.
\end{align}

With $\Sym_n^+$ and Artin monomials replaced by $\qsym_n^+$ and ABB monomials, forest polynomials satisfy a statement in the spirit of~\eqref{eq:schubert_property}.
The work of Aval--Bergeron--Bergeron \cite{ABB04} implies the following decomposition
\begin{align}
  \label{eq:abb}
\bQ[\alpxn{n}]=\bQ\{\alpx^{\bfc}\suchthat \bfc\in \nvectsn{n} \text{ satisfying } \sum_{1\leq i\leq j}c_{n+1-i}\leq j-1  \text{ for } j=1,\dots,n\}\oplus \qsym_n^+.
\end{align}
The monomials in the first summand in~\eqref{eq:abb} are the ABB monomials from before.
Let $\abb{n}$ denote the set of $\bN$-vectors that determine ABB monomials.
One can thus ask for representatives in $\bQ[\alpxn{n}]/\qsym_n^+$ in terms of their expansion in the basis of ABB monomials.
The following theorem  serves to emphasize the parallel between Schubert polynomials and forest polynomials.
\begin{theorem}
\label{th:main_4}
For $\bfc \in \nvectsn{n}$ the following holds:
\begin{align*}
  \forestpoly_{\bfc}\mod \qsym_n^+=\left\lbrace\begin{array}{ll}
  \forestpoly_{\bfc}  & \bfc\in \abb{n}\\
  0 & \bfc\notin \abb{n}. \end{array}\right.
\end{align*}
\end{theorem}
Theorems~\ref{th:main_3} and \ref{th:main_4} together imply our combinatorial rule for $a_w$ in Theorem~\ref{th:aw_as_parking}. The proof of Theorem~\ref{th:main_4} occupies Appendix~\ref{app:sign_reversing}.

\medskip

\noindent\textbf{Outline of article}:
Section~\ref{sec:bg} introduces the necessary background. Section~\ref{sec:forest_polynomials} introduces our chief combinatorial objects: indexed forest and forest polynomials. It describes in particular the slide expansion, a recurrence relation, as well as the principal specialization.
Section~\ref{sec:forest_and_qsym} sets the stage for Theorem~\ref{th:main_4} by introducing `dual' forest polynomials.
Section~\ref{sec:wf} introduces the correspondence $\wfcorr$ (as well the $\Omega$-parking procedure) which forms the crux of our work. We then prove several elementary yet helpful properties of the resulting equivalence relation $\wfequiv$. This section culminates with our central result which relates DS and $\Omega$-parking: Theorem~\ref{th:slide_to_parking}.
Section~\ref{sec:applications} provides some applications of the machinery developed leading up to it.
As a special case of our general result, we derive our interpretation for $a_w$ in Section~\ref{subsec:combinatorial_interpretation}. 
In Section~\ref{subsec:mult} we describe a `shuffle rule' to \emph{nonnegatively} describe the product of two  forest polynomials in the basis of forest polynomials. 
Section~\ref{subsec:mme} leverages the connection between DS and $\qsym_n^+$ to propose \emph{multivariate} mixed Eulerian numbers that readily recover both ordinary and remixed versions.
We conclude in Section~\ref{sec:the_rest} with some remarks suggesting avenues of exploration and hinting at upcoming work.

\section{Background}
\label{sec:bg}
We set up essential notation. To avoid burdening the reader, some notions will be introduced when they are first needed. Given integers $a<b$, we let $[a,b]$ denote the interval $\{a,a+1,\dots,b-1,b\}$ in $\bZ$.
If $a=1$, we simply write $[b]$.
In particular $[0]$ is the empty set.

\subsection{$\bN$-vectors}
We define an \emph{$\bN$-vector} $\bfc$ to be a sequence of nonnegative integers $(c_i)_{i\in \bZ}$ where all but finitely many $c_i$ are $0$. The \emph{support} $\supp(\bfc)$ of  $\bfc$ is the set of indices $i\in \bZ$ such that $c_i>0$.
The \emph{weight} of $\bfc$ is $|\bfc|\coloneqq \sum_i c_i$.
We denote the set of $\bN$-vectors by $\nvects$.

Recall from the introduction that $\nvectsp$ denotes the set of positively-supported $\bN$-vectors, i.e. $\bfc$ such that $\supp(\bfc)\subset \bZ_{\geq 1}$.
These are also the $\bN$-vectors that will concern us for the most part.
Given a positive integer $n$, recall further that $\nvectsn{n}$ is the set of $\bN$-vectors $\bfc$ satisfying $\supp(\bfc)\subseteq[n]$. We will write elements in $\nvectsn{n}$ as finite sequences $(c_1,\dots,c_n)$, and refer to them as \emph{weak compositions}.

Certain weak compositions play an important role for us. 
The set $\abb{n}\subset \nvectsn{n}$ comprises weak compositions $(c_1,\dots,c_n)$  such that
\[
\sum_{1\leq i\leq j}c_{n+1-i}\leq j-1 \text{ for all } j=1,\dots,n.
\]
For instance we have
\begin{align*}
  \abb{1}=\{(0)\},\,
  \abb{2}=\{(1,0),(0,0)\},\,
  \abb{3}=\{(2,0,0),(1,1,0),(1,0,0),(0,1,0),(0,0,0)\}.
\end{align*}
Note that $|\abb{n}|=\cat_n$.
Let $\abbtop{n}$ denote the set of ABB compositions in $\abb{n}$ of maximal weight, i.e. those of weight $n-1$.
Note that $|\abbtop{n}|=\cat_{n-1}$.

Omitting the zeros from $\bfc\in\nvects$  results in a \emph{strong composition} $\flatten(\bfc)$ of the same weight.
Henceforth we simply say composition when talking about strong compositions.
The entries in a composition $\alpha$ are called its \emph{parts}, and the number of parts is called its \emph{length}, which is denoted by $\ell(\alpha)$.
The unique composition of length and weight both equal to $0$ will be denoted by $\varnothing$.

\subsection{Permutations and Lehmer codes}
Let $S_{\bZ}$ denote permutations of $\bZ$ that leave all but finitely many $i\in \bZ$  fixed.
Given $w\in S_\bZ$, define its \emph{Lehmer code} $\lcode(w)\coloneqq(c_i)_{i\in \bZ}$ where $c_i=|\{j>i\suchthat w(i)>w(j)\}|$.
Note that $w\mapsto \lcode(w)$ sets up a bijection between $S_\bZ$ and $\nvects$.

Given $w\in S_{\bZ}$, let $\supp(w)\coloneqq \{i\in \bZ\suchthat w(i)\neq i\}$.
We then define $S_{+}$ to be the set of permutations $w$ satisfying $\supp(w)\subset \bZ_{\geq 1}$ .
It is immediate that $\lcode$ restricts to a bijection between $S_{+}$ and $\nvectsp$.
It is natural to inquire what $\nvectsn{n}$ is in bijection with.
Recall that a \emph{descent} of a permutation is an index $i$ such that $w(i)>w(i+1)$.
Then the subset of $S_{+}$ comprising permutations whose last descent is at most $n$, i.e. $w(i)<w(i+1)$ for $i>n$, is in bijection with $\nvectsn{n}$.
Finally for a positive integer $n$, we define $S_n$ to be the set of permutations $w$ satisfying $\supp(w)\subseteq [n]$.
The corresponding codes are what we called $\artin{n}$ earlier, i.e. weak compositions $(c_1,\dots,c_n)$ satisfying $c_i\leq n-i$ for $1\leq i\leq n$.

Note that $S_{\bZ}$ is generated by the adjacent transpositions $(s_i)_{i\in \bZ}$.
Given $w\in S_{\bZ}$, recall that a reduced word for $w$ is a word $i_1\dots i_k$ of minimal length such that $s_{i_1}\cdots s_{i_k}=w$, and that $\red(w)$ denotes the set of reduced words of $w$.

\subsection{(Quasisymmetric) Polynomials}
Let $\alpx=\{x_i\suchthat i\in \bZ\}$ denote a totally ordered set of commuting indeterminates where the letters are sorted according to their indices under the natural order on $\bZ$.
Let $\alpxplus=\{x_i\suchthat i\geq 1\}$.
Given $n\in \bZ_{\geq 0}$, we let $\alpxn{n}=\{x_i\suchthat 1\leq i\leq n\}$.
We let $\bQ[\alpxplus]$ and $\bQ[\alpxn{n}]$ denote the polynomial rings in the appropriate sets of variables respectively. In particular when $n=0$, this ring is simply $\bQ$.

A polynomial $f(x_1,\dots,x_n)$ in $\alpxn{n}$ is said to be \emph{quasisymmetric} if for any composition $(\alpha_1,\dots,\alpha_k)$ where $k\leq n$, the coefficient of $x_{i_1}^{\alpha_1}\cdots x_{i_k}^{\alpha_k}$ and $x_{j_1}^{\alpha_1}\cdots x_{j_k}^{\alpha_k}$
are equal whenever $1\leq i_1< \cdots < i_k\leq n$ and $1\leq j_1< \cdots < j_k\leq n$.
The $\bQ$-vector space consisting of quasisymmetric polynomials is in fact a ring. A distinguished linear basis for this space is given by Gessel's \emph{fundamental quasisymmetric polynomials} \cite{Ges84}.
These polynomials are indexed by compositions; we let $\fund_{\alpha}(\alpxn{n})$ denote the fundamental quasisymmetric polynomial indexed by the composition $\alpha$. 

In the interest of space, we refer the reader to \cite[Chapter 7]{St99} for an in-depth treatment of this subject.
We merely note a fact of value to us: a standard source for quasisymmetric polynomials/functions in combinatorial literature is Stanley's theory of $P$-partitions (and more generally $(P,\omega)$-partitions). 
Given a naturally labeled poset $P$, we have a $P$-partition generating function $K_P(\alpxplus)$ which turns out to equal a sum of fundamental quasisymmetric functions $\fund_{\alpha}(\alpxplus)$, one for each \emph{linear extension} of $P$.
Truncating this to the alphabet $\alpxn{n}$ gives a statement involving quasisymmetric polynomials.

\subsection{Slide polynomials and Schubert polynomials}
\label{subsec:back_slide}

We now proceed to define two distinguished basis for $\bQ[\alpxplus]$. 
Let $\bZ^*$ denote the set of words in the alphabet $\bZ$.
We denote the empty word by $\varepsilon$.
Given $\mathbf{i}=i_1\cdots i_k\in \bZ^*$, let $\compatible(\mathbf{i})$ denote the set of \emph{compatible sequences}:
\[
\{(a_1\geq \dots \geq a_k)\suchthat  1\leq a_j\leq i_j \text{ for  } 1\leq j \leq k \text{ and $a_j>a_{j+1}$ if $i_j>i_{j+1}$ for $1\leq j\leq k-1$}\}.
\]
Now define the \emph{slide polynomial} $\slide({\mathbf{i}})$ by
\begin{align}
  \label{eq:slide_def}
\slide({\mathbf{i}})=\sum_{(a_1,\dots,a_k)\in \compatible(\mathbf{i})}x_{a_1}\cdots x_{a_k}.
\end{align}
In particular we define $\slide(\varepsilon)\coloneqq 1$.

It is the case that distinct $\mathbf{i}$ can result in the same $\slide(\mathbf{i})$.
For instance, check that $\compatible(422)=\compatible(423)$.
Additionally, it may be that $\compatible(\mathbf{i})$ is empty, in which case $\slide(\mathbf{i})=0$.
We isolate a subset of the $\slide(\mathbf{i})$ as $\mathbf{i}$ varies over $\bZ^*$ with the property that any nonzero $\slide(\mathbf{i})$ equals an element in our subset.

Given $\bfc\in \nvectsp$, let $W_{\bfc}$ be the unique nonincreasing word of length $|\bfc|$ with $c_i$ instances of $i$. For instance, if $\bfc=(0,3,2,0,3)$, then $W_{\bfc}=55533222$. 
The \emph{slide polynomial} $\slide_{\bfc}$ is defined to be $\slide(W_{\bfc})$.
The revlex leading term of $\slide_{\bfc}$ is $\alpx^{\bfc}\coloneqq \prod_{i}x_i^{c_i}$.
This property implies that the set $\{\slide_{\bfc}\}_{\bfc\in\nvectsp}$ is a $\bQ$-basis for $\bQ[\alpxplus]$  \cite{AssSea17}.

\begin{example}\label{ex:backslide}
  Let $\bfc=(1,0,2)$.
  Then $W_{\bfc}=331$, and we have
  \[
  \compatible(331)=\{(a_1\geq a_2>a_3)\suchthat 1\leq a_1,a_2\leq 3, 1\leq a_3\leq 1\}.
  \]
  This implies that $\slide_{\bfc}=x_1x_3^2+x_1x_2^2+x_1x_2x_3$.
\end{example}

Having defined slide polynomials, we are ready to introduce the Schubert polynomials of Lascoux--Sch\"utzenberger \cite{Las82}. 
The \emph{Schubert polynomial} $\schub{w}$ for $w\in S_{+}$ is defined here via the  celebrated description due to Billey--Jockusch--Stanley \cite{Bil93}:
\begin{align}
  \label{eq:schub_def}
\schub{w} =\sum_{\mathbf{i}\in \red(w^{-1})}\slide(\mathbf{i}).
\end{align}
The revlex leading term of $\schub{w}$ is equal to $\alpx^{\lcode(w)}$ implying  that the set $\{\schub{w}\}_{w\in S_+}$ is a $\bQ$-basis for $\bQ[\alpxplus]$.
In fact we have (see~\cite[2.5.2]{Man01}): 
\begin{theorem}\label{th:schub_basis}
	For any positive integer $n$, the family of Schuberts polynomials $\schub{w}$ indexed by all $w\in S_+$ such that the last descent in $w$ is in position at most $n$ forms a basis for $\bQ[\alpxn{n}]$.
	Furthermore
	\[
	\bQ\{\schub{w}\suchthat w\in S_n\}=\bQ\{\alpx^{\bfc}\suchthat \bfc \in \artin{n}\}.
	\]
\end{theorem}

\begin{example}\label{ex:schub}
  Consider $w=14253\in S_{5}$ in one line notation.
  Then $\red(w^{-1})=\{423,243\}$.
  We leave it to the reader to check that
   $\slide(423)=\slide_{(0,2,0,1)}$ and $\slide(243)=\slide_{(1,2)}$.
 By~\eqref{eq:schub_def} 
 \begin{align*}
 \schub{14253}=\slide_{(0,2,0,1)}+\slide_{(1,2)} =x_2^2x_4+x_2^2x_3+x_1^2x_4+x_1^2x_3+x_1^2x_2+x_1x_2x_4+x_1x_2x_3+x_1x_2^2.
\end{align*}
\end{example}

\section{Forest polynomials}
\label{sec:forest_polynomials}

In this section we define forest polynomials, that we index by what we call indexed forests. These are defined as an extension of the notion of binary trees, so we first  briefly recall standard notions about those. 

A (plane) \emph{binary tree} is given recursively by being either the empty set, or a node with a left subtree and a right subtree.
 Binary trees with $n$ nodes are counted by Catalan numbers. Complete binary trees consist of replacing the empty set in the definition by a ``leaf''.
A complete binary tree with $n$ nodes has $n+1$ leaves. 
The \emph{canonical labeling} of a binary tree with $n$  nodes corresponds to labeling these nodes from $1$ to $n$ in inorder. 
By aligning all $n+1$ leaves on an integer line, the canonical order corresponds naturally to labeling the unit segments by $1$  through $n$ reading from left to right, cf. Figure~\ref{fig:canonical_labeling}.

\begin{figure}[!ht]
\centering
\includegraphics[width=0.4\linewidth]{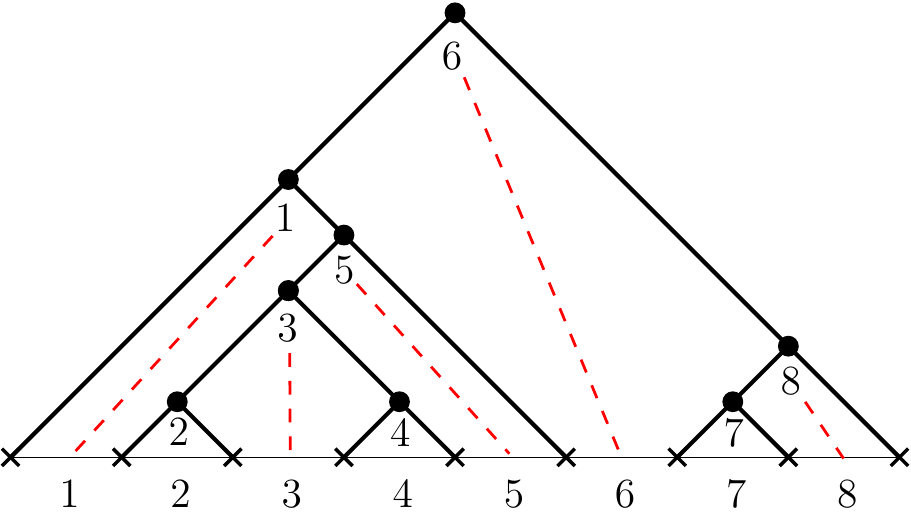}
\caption{A binary tree and its canonical labeling
\label{fig:canonical_labeling}}
\end{figure}

\subsection{Indexed forests}

\begin{definition}Let $S$ be a finite set of integers. It decomposes uniquely as $S=I_1\sqcup\ldots\sqcup I_k$, where each $I_j$ is a maximal subset of consecutive integers in $S$. An \emph{indexed forest $F$ on $S$} is the data of a binary tree $T_j$ with $|I_j|$ internal nodes for any $j\in[k]$.
\end{definition}

In this definition $S$ is the {\em support} of $F$, and we write $S=\supp(F)$. 
We will simply talk of an indexed forest if we do not require to specify the support. 
Let $\IF(S)$ be the set of indexed forests with support $S$, $\IF$ be the set of all indexed forests, and $\IF_+$ those with support consisting of positive integers.

   We extend the canonical labeling to indexed forests by labeling the internal nodes of $T_j$ by the elements of $I_j$ in inorder: by definition, $|I_j|=|T_j|$, so this is just a shift of the usual labeling.
  Pictorially, this labeling is best read as in   Figure~\ref{fig:indexed_forest} (ignoring the red labels for the moment): we extend the case of binary trees, where leaves are naturally aligned on the horizontal axis, in which the labels in $I_j$ are  matched with the inner nodes of $T_j$.   
 Given $u\in T_j$, we let $\interval(u,F)$ be the set of canonical labels of all internal nodes in the subtree rooted at~$u$.

\begin{figure}[!ht]
\centering
\includegraphics[width=0.8\linewidth]{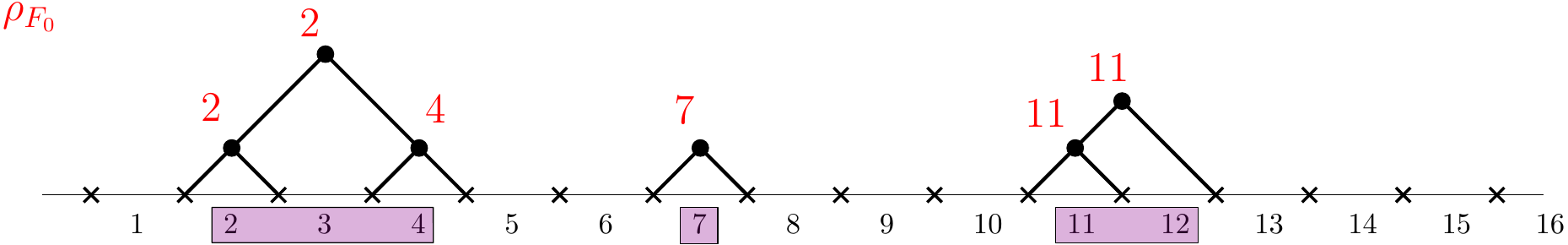}
\caption{An indexed forest $F_0$ with support $\{2,3,4,7,11,12\}$
\label{fig:indexed_forest}}
\end{figure}

Let $\internal(F)$ be the set of \emph{internal} nodes of $F$. Let $|F|$ be its size $|S|$. Let $\can_F:\internal(F)\to S$ be the canonical labeling defined above. We call a node $v\in \internal(F)$ \emph{terminal} if both its children are leaves. 
Let $\terminal(F)$ denote the set of terminal nodes in $F$.
The  \emph{left support} $\lsupp(F)$ of $F$ is the set of labels of nodes whose left child is a leaf. 
The indexed forest in Figure~\ref{fig:indexed_forest} has size $6$, it has 4 terminal nodes, its support is $\supp(F)=\{2,3,4\}\sqcup \{7\}\sqcup\{11,12\}$ while its left support is $\lsupp(F)=\{2,4,7,11\}$.\smallskip

Define $\rho_F:\operatorname{IN}(F)\to \bZ$ by
\[
\rho_F(v)\coloneqq \text{the node label reached by following left edges down from $v$}.
\]
As an example, see red labels in Figure~\ref{fig:indexed_forest}.
Now define $\bfc(F)=(c_i)_{i\in \bZ}\in \nvects$ by
\[
c_i=|\{v\in \internal(F)\suchthat \rho_F(v)=i\}|.
\]
\begin{example}
For $F$ in Figure~\ref{fig:indexed_forest} we get $\bfc(F)=(0,\textcolor{magenta}{2,0,1,0},0,\textcolor{magenta}{1,0},0,0,\textcolor{magenta}{2,0,0},0,0,\ldots)\in \nvectsp$.
Note that the highlighted contiguous subsequences are in fact elements of $\abbtop{i}$ for various $i\geq 2$.
The intermediary (non-highlighted) $0$s of course correspond to the only element in $\abb{1}$.
\end{example}

It is clear that any element of $\nvects$ can be written uniquely as a concatenation of elements in $\abbtop{i}$ for varying $i\in \bZ_{\geq 1}$.
In view of the folklore bijection between elements of $\abbtop{i}$ and complete binary trees with $i-1$ internal nodes, we get the following:
\begin{proposition}
\label{prop:forests_and_codes}
The correspondence $F\mapsto \bfc(F)$ is a bijection between $\IF$ and $\nvects$. It restricts to a bijection between $\IF_+$ and $\nvectsp$.
\end{proposition}
Let us denote the correspondence in the opposite direction as $\bfc\mapsto F(\bfc)$.
Consequently, objects indexed by $\nvectsp$ may be indexed by $\IF$, and vice versa.
We will employ this correspondence at will and without prior warning.
We will on occasion omit parentheses, commas, and trailing zeros in writing our $\bN$-vectors for the sake of clarity.

\subsection{Forest polynomials}
\label{subsec:forest}

We are ready to introduce forest polynomials. 

\begin{definition}
  \label{def:forest} Let $F\in \IF$. The \emph{forest polynomial} $\forestpoly_F\in\bQ[\alpxplus]$  is defined as
\begin{align}
  \label{eq:backforest}
\forestpoly_F=\sum_{\kappa} \prod_{v\in \internal(F)}x_{\kappa(v)}
\end{align}
where the sum is over all labelings $\kappa:\internal(F)\to\bZ_{\geq 1}$ that are bounded above by $\rho_F$, weakly increasing down left edges and strictly increasing down right edges. In particular, if $|F|=0$, then we set $\forestpoly_F\coloneqq 1$.
\end{definition}

Note that if $F\notin\IF_+$, then $\forestpoly_F=0$: indeed, in this case $\rho_F(v)\leq 0$ for the smallest element $v$ in inorder, and thus $\kappa(v)\leq \rho_F(v)$ cannot be satisfied since $\kappa(v)\geq 1$ by definition. 
Also it is clear from this definition that if $F$ is composed of indexed trees $T_1$ through $T_k$ for some $k\geq 1$, then the following multiplicative property holds:
\begin{align}
  \label{eq:multiplicative}
  \forestpoly_F=\prod_{1\leq i\leq k}\forestpoly_{T_i}.
\end{align}

\begin{figure}
\centering
\includegraphics[width=0.4\linewidth]{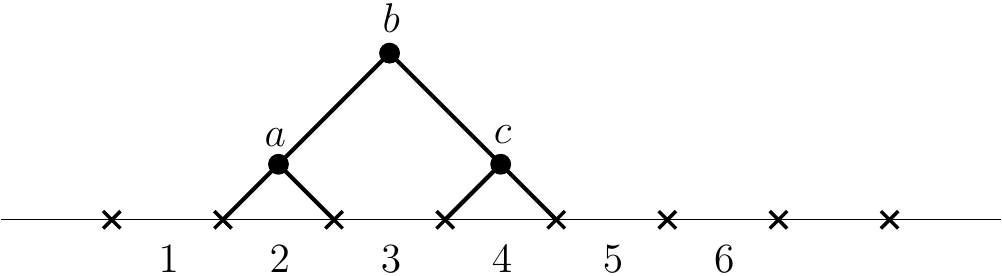}
\caption{An indexed tree $T_0$ with $\bfc(T_0)=(0,2,0,1,0,\dots)$}
\label{fig:forest_0201}
\end{figure}

\begin{example}
\label{exa:pT0}
Consider the indexed tree $T_0$ in Figure~\ref{fig:forest_0201}. It is also the leftmost indexed tree in the indexed forest in Figure~\ref{fig:indexed_forest}. Then
\begin{align}
\label{eq:running_example}
\forestpoly_{T_0}=\sum_{\substack{{2}\geq a\geq b>0\\{4}\geq c>b>0}}x_ax_bx_c
=x_{2}^{2} x_{4} + x_{1} x_{2} x_{4} + x_{1}^{2} x_{4} + x_{2}^{2} x_{3} + x_{1} x_{2} x_{3} + x_{1}^{2} x_{3}  + x_{1}^{2} x_{2}+ x_{1} x_{2}^{2}.
\end{align}
which coincides with the Schubert polynomial $\schub{14253}$, cf. Example~\ref{ex:schub}. 
By the multiplicative property~\eqref{eq:multiplicative}, for $F_0$ in Figure~\ref{fig:indexed_forest} we get \[
 \forestpoly_{F_0}=\forestpoly_{T_0}\left(\sum_{i=1}^7x_i\right)\left(\sum_{1\leq i\leq j\leq 11}x_ix_j\right).
 \]
\end{example}

We now begin to establish some basic properties of forest polynomials.
The first of these is the fact that this family of polynomials, like Schubert and  slide polynomials, is a basis for the polynomial ring.
This in turn motivates the study of transition matrices between these various bases, and the Schubert into forest expansion is an important cog in our setup.
The reader should compare the following statement to Theorem~\ref{th:schub_basis}.

\begin{proposition}
\label{prop:forests_as_basis}
The family $\left(\forestpoly_F\right)_{F\in\IF_+}$ is a basis for $\bQ[\alpxplus]$. 
More precisely, for any positive integer $n$ the family $\{\forestpoly_F\}$ as $F$ ranges over all indexed forests such that $\lsupp(F)\subseteq [n]$ is a basis for $\bQ[\alpxn{n}]$.
 Furthermore
	\[
	\bQ\{\forestpoly_F\suchthat \supp(F)\subseteq [n-1]\}=\bQ\{\alpx^{\bfc}\suchthat \bfc \in \abb{n}\}.
	\]
\end{proposition}
\begin{proof}
Both statements follow easily from the fact that $\alpx^{\bfc(F)}$ is the revlex leading term of~$\forestpoly_{F}$.
\end{proof}

Recall that our definition of slide polynomials used a word rather than a weak composition. 
While different words can lead to the same slide polynomial, this alternative indexing has its benefits and gives added flexibility in manipulating slide polynomials.
We now describe a similar perspective on forest polynomials by introducing a family of flags that contains $\rho_F$.
To this end we need local binary search forests.

\subsection{Local binary search forests} We recall the alphabet $\Zaug$ introduced in \cite{NT23}. It is the ordered alphabet with letters $\lf{i}{j}$ where $i\in \bZ$ and $j\in \bP$. We have a linear order on $\bZ\sqcup \Zaug$ given by $i<\lf{i}{1}<\lf{i}{2}<\lf{i}{3}<\cdots<i+1$ for all $i$. The \emph{value} of $\lf{i}{j}$ by $\mathrm{val}(\lf{i}{j})=i$.

Let $F\in \IF$. We denote by $\lbs(F)$ the set of all {\em injective} functions $\rho:\internal(F)\to\Zaug$ that satisfy the following constraints:
\begin{itemize}
\item $\rho(v)<\rho(u)$ if $v$ is the left child of $u$.
\item $\rho(u)<\rho(v)$ if $v$ is the right child of $u$.
\item $\mathrm{val}(\rho(u))\in \interval(u,F)$ for all $u\in\internal(F)$.
\end{itemize}
$\lbs(F)$ is the set of ``local binary search'' labelings of $F$: this usually refers in the literature to labelings satisfying only the first two conditions above. 
We denote by $\lbs$ the union of $\lbs(F)$ over all $F\in\IF$, that is the set of all pairs $(F,\rho)$.

If $\rho$ is a function from $\internal(F)$ to $\bZ$ instead of $\Zaug$, we can replace labels of vertices with a fixed $\rho$-value, say $i$, by $\lf{i}{1},\lf{i}{2},\ldots$ in inorder. 
This transforms any such function $\rho$  into an injective function from $\internal(F)$ to $\Zaug$, and we will silently identify the two functions.

In this manner we have that $\rho_F\in \lbs(F)$.
 Now for any $\rho\in \lbs(F)$, define $\forestpoly(F,\rho)$ following  Definition~\ref{def:forest}, except that the upper bounds for the labelings $\kappa$ are given by $\rho$ instead of $\rho_F$.
 It turns out that this does not modify the set of polynomials.

\begin{proposition}
\label{prop:betaP_is_betaF}
For any $F\in\IF$ and $\rho\in \lbs(F)$, we have $\forestpoly(F,\rho)=\forestpoly_F$.
\end{proposition}
\begin{proof}
First, note that the values of $\rho$ are necessarily bounded below by $\rho_F$ since the minimum of $\interval(u,F)$ is $\rho_F(u)$.
We thus need only show that any labeling $\kappa:\internal(F)\to \bP$ bounded above by $\rho$ is in fact bounded by $\rho_F$.
This is true for $u$ a terminal node, since $\mathrm{val}(\rho(u))=\rho_F(u)$ in this case.
Now assume $u$ is not terminal. If it has a left child $v$, so that $\rho_F(u)=\rho_F(v)$, then by induction we have $\kappa(v)\leq \rho_F(v)$ and since $\kappa(u)\leq \kappa(v)$ we obtain $\kappa(u)\leq \rho_F(u)$ as wanted.
 Otherwise $u$ has no left child but has a right child $v$. In this case $\rho_F(u)=\rho_F(v)-1$ and $\kappa(u)<\kappa(v)$. By induction $\kappa(v)\leq \rho_F(v)$ and thus $\kappa(u)\leq \rho_F(u)$ which completes the proof.
\end{proof}

\subsection{From forest polynomials to slide polynomials}
\label{subsec:backforests_into_backslides}
An indexed forest $F$ naturally determines a \emph{forest poset}, by interpreting the parent-child relation as a cover relation in the Hasse diagram on $\internal(F)$ .
The restriction map $\rho_F$ is determined by the leaves, and this in turn places constraints on the labels on $\internal(F)$.
This entire framework, including the weak and strict inequalities on left and right edges respectively, is evocative of Stanley's theory of $(P,\omega)$-partitions \cite{StThesis,GesSurvey}.

As the authors explored, drawing motivation from ibid. and recent work of Assaf--Bergeron \cite{AssBer20}, forest polynomials can indeed be described as generating polynomials for a class of $(P,\Phi)$-partitions; see \cite[Section 3.3]{NT23}.
It is a consequence of the theory developed in \cite{NT23} that forest polynomials permit an expansion in terms of slide polynomials akin to how a $P$-partition generating function/polynomial expands in terms of Gessel's fundamental quasisymmetric functions/polynomials.
We need some more notation to describe it.

A \emph{decreasing labeling} of an indexed forest $F$ is a labeling of $\internal(F)$  with labels drawn from $\{1,\dots,|F|\}$ so that the labels always decrease as one moves away from root nodes.
We denote the set of decreasing labelings of $F$ by $\decreasing(F)$.
Such labelings  give \emph{linear extensions} of the forest poset underlying $F$.
For $\ell\in \decreasing(F)$, let $\mathbf{i}_{\ell}\in \bZ^*$ be the word obtained by recording the $\rho_F$ values as one visits internal nodes in accordance with the linear extension $\ell$.
We then have:
\begin{proposition}[{\cite[Equation 3.4]{NT23}}]
\label{prop:forests_into_slides}
Given $F\in \IF $, we have the following expansion for $\forestpoly_F$ in terms of slide polynomials:
\begin{align}
\label{eq:forests_into_slides}
\forestpoly_F=\sum_{\ell\in \decreasing(F)}\slide(\mathbf{i}_{\ell}).
\end{align}
\end{proposition}

\begin{example}
\label{ex:bf_into_bs}
	Consider the indexed tree $T_0$ in Figure~\ref{fig:bf_into_bs_example}.
	The $\rho_{T_0}$ values are in red. On the right are the two labelings in $\decreasing(T_0)$. Reading the $\rho_{T_0}$ values according to the two decreasing labelings gives the words $422$ and $242$ respectively. Proposition~\ref{prop:forests_into_slides} then says that
	\begin{align*}
	\forestpoly_{T_0}&=\slide(422)+\slide(242)=\slide_{(0,2,0,1)}+\slide_{(1,2)}\\
	&=x_{2}^{2} x_{4} + x_{1} x_{2} x_{4} + x_{1}^{2} x_{4} + x_{2}^{2} x_{3} + x_{1} x_{2} x_{3} + x_{1}^{2} x_{3}  + x_{1}^{2} x_{2}+ x_{1} x_{2}^{2}.
	\end{align*}
This matches the monomial expansion from Example~\ref{exa:pT0}.
\end{example}

\begin{figure}[!ht]
\includegraphics[width=0.8\linewidth]{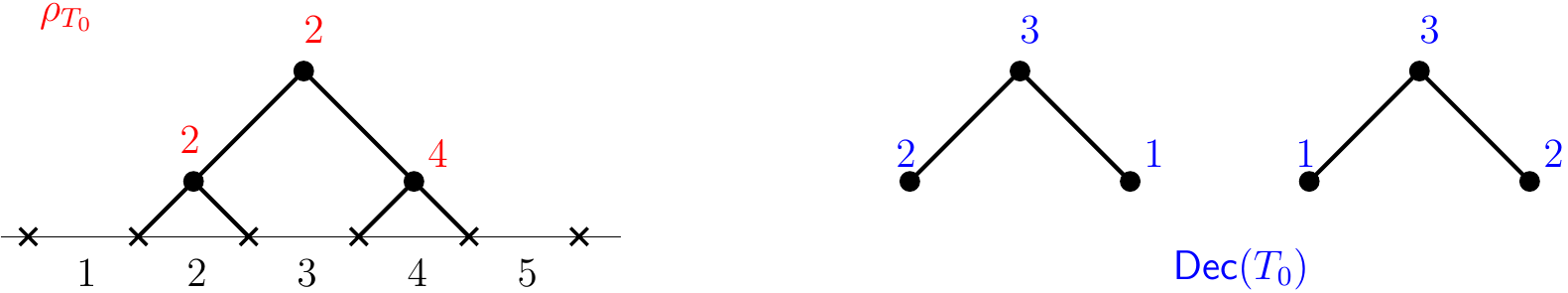}
\caption{The indexed tree $T_0$ with its two decreasing labelings}
\label{fig:bf_into_bs_example}
\end{figure}

More generally, let $\rho\in \lbs(F)$ and $P=(F,\rho)$. 
We abuse notation and set $\forestpoly(P)\coloneqq \forestpoly(F,\rho)$.
Then
\begin{align}
\label{eq:forests_into_slides_generalized}
\forestpoly(P)=\sum_{\ell\in \decreasing(F)}\slide(\mathbf{i}_{\ell}),
\end{align} 
where the word $\mathbf{i}_{\ell}$ records the values of $\rho$ following the linear extension $\ell$. We have $\mathbf{i}_{\ell}\in\injwords{\Zaug}$ which by definition is the set of words in $\Zaug$ with pairwise distinct letters.

The result follows directly from the theory in~\cite{NT23}.
\smallskip

A particular case of forest polynomials deserves special mention. The expansion in Proposition~\ref{prop:forests_into_slides} has one term precisely when our indexed forest $F$ has a single linear extension. This happens precisely when $F$ is an indexed \emph{linear} tree, i.e. an indexed tree which forms a path.

\begin{proposition}\label{prop:fundamentals_as_forests}
$F\in \IF_+$ is an indexed linear tree if and only if $\supp(\bfc(F))$ is an interval in $\bZ_{\geq 1}$, if and only if $\lsupp(F)$ is an interval in $\bZ_{\geq 1}$.

In that case, write $\bfc(F)=(0^j,\alpha_1,\dots,\alpha_k)$ where $\alpha_i>0$ for $i=1,\dots,k$ and $j\geq 0$.
Then
\begin{align}
\label{eq:forest_contain_fundamentals}
  \forestpoly_F=\fund_{\alpha}(x_1,\dots,x_{k+j}).
\end{align}
\end{proposition}

\begin{proof}
The equivalent conditions for $F$ being an indexed linear tree are readily checked. To obtain~\eqref{eq:forest_contain_fundamentals}, notice that the weak and strict inequalities along left and right edges of a linear tree ensure that Definition~\ref{def:forest} is exactly the definition of a fundamental quasisymmetric polynomial. Alternatively, one can notice that the unique slide polynomial occurring in \eqref{eq:forests_into_slides} is indeed   the desired fundamental quasisymmetric polynomial.
\end{proof}

\begin{example}\label{ex:linear_trees_are_fundamental}
  Consider the three indexed linear trees in Figure~\ref{fig:linear} from left to right. The preceding discussion implies that the corresponding forest polynomials equal $\fund_{(2,2,1)}(x_1,x_2,x_3,x_4)$, $\fund_{(2)}(x_1,x_2)$, and $\fund_{(1,1)}(x_1,x_2,x_3)$ respectively.
  The last two could be written as $h_2(x_1,x_2)$ and $e_3(x_1,x_2,x_3)$ respectively, where the $h$ and $e$ denote the \emph{complete homogeneous symmetric polynomial} and \emph{elementary symmetric polynomial}.
\end{example}

\begin{figure}[]
\includegraphics[width=0.9\linewidth]{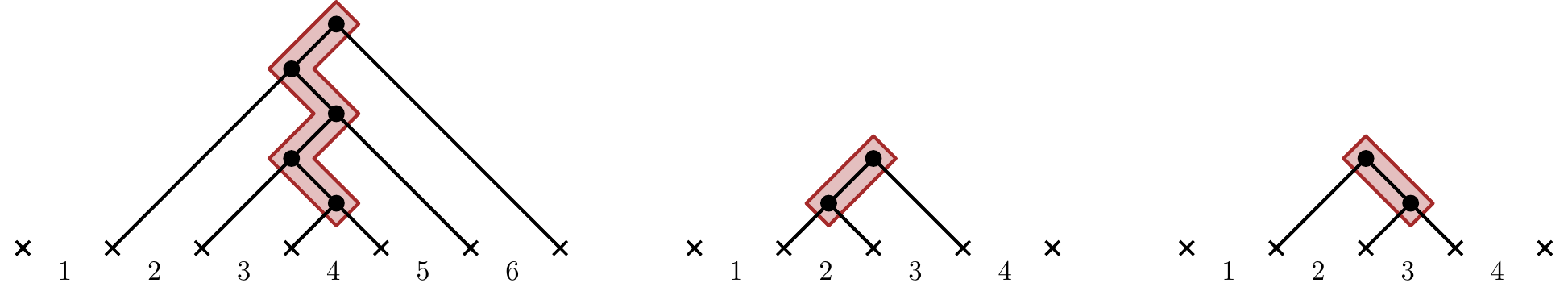}
\caption{Indexed forests with left support an interval}
\label{fig:linear}
\end{figure}

\subsection{A recurrence and specializations}
\label{subsec:recurrence}
We now proceed to describe a recurrence for forest polynomials which is useful in more ways than one. It aids fast computations of the $\forestpoly_F$ and brings forth another connection to linear extensions of $F$.

We discuss two operations on indexed forests that will allow us to describe this recurrence. The first is the \emph{shift} map $\gamma$ which, given an indexed forest $F$, simply shifts $\supp(F)$ one unit to the right.
We abuse notation and define $\gamma^{i}(\forestpoly_F)=\forestpoly_{\gamma^{i}(F)}$ for $i\in \bZ$, making note that if $\supp(\gamma^{i}(F))\nsubseteq \bZ_{>0}$, then $\forestpoly_{\gamma^{i}(F)}$ equals $0$.

The second operation is \emph{trimming}; it corresponds to removing a terminal node. 
Explicitly, say $v\in \terminal(F)$ lies in an indexed tree $T$ with support $[a,b]$ that is a constituent of $F$.
We obtain $\trim{v}(F)$ from $F$ by replacing $T$ with $T'= T\setminus \{v\}$ supported on $[a,b-1]$, and leaving the other trees in $F$ unaltered. 
In the special case $a=b$, implying that $T$ had one node to begin with, we simply delete this node from our forest.
As an example, observe that the indexed trees on the center and the right in Figure~\ref{fig:linear} are precisely those obtained by trimming terminal nodes in $T_0$ from Figure~\ref{fig:forest_0201}.

\begin{lemma}
\label{lem:leaf_indexed}
Given $F\in \IF_+$ we have
\begin{align}
\label{eq:leaf_indexed}
\forestpoly_F=\forestpoly_{\gamma^{-1}(F)}+\sum_{v\in \terminal(F)}x_{\rho_F(v)}\,\forestpoly_{\trim{v}(F)}.
\end{align}
\end{lemma}
\begin{proof}
Consider labelings $\kappa$ that contribute monomials to $\forestpoly_F$ as per Definition~\ref{def:forest}.
Say the terminal nodes in $F$ are $v_1$ through $v_k$ from left to right for some $k\geq 1$.
The labelings $\kappa$ satisfying $\kappa(v_i)< \rho_F(v_i)$ for all $1\leq i\leq k$ contribute to $\forestpoly_{\gamma^{-1}(F)}$
Indeed, decrementing all entries in $\rho_F$ by $1$ is the same as shifting the support of the indexed forest one unit to the left.
So let us consider the contributions of labelings wherein a terminal node has the same label as its $\rho_F$-value.

Fix an $i\in [k]$ and consider labelings $\kappa$ such that $\kappa(v_i)=\rho_F(v_i)$ but $\kappa(v_j)<\rho_F(v_j)$ for all $j>i$. The contribution of such labelings is then seen to equal $x_{\rho_F(v_i)}\forestpoly_{\trim{v_i}(F)}$.
Summing these contributions over all $1\leq i\leq k$ implies the claim.
\end{proof}
One can rewrite this recurrence by indexing forest polynomials by  $\bN$-vectors, and trade clarity for brevity.
In the case of indexed linear trees, in view of Proposition~\ref{prop:fundamentals_as_forests}, the recurrence  in Lemma~\ref{lem:leaf_indexed} is then an easy one for fundamental quasisymmetric polynomials \cite[Proposition 2.1]{ABB04} which plays a useful role in loc. cit..

\begin{example}
\label{ex:forest_0201_again}
For $T_0$ from Figure~\ref{fig:forest_0201},  Lemma~\ref{lem:leaf_indexed} says
\[
\forestpoly_{T_0}=\forestpoly_{\gamma^{-1}(T_0)}+x_4\forestpoly_{T_2}+x_2\forestpoly_{T_3},
\]
where $T_2,T_3$ are the trees on the center and right of Figure~\ref{fig:linear}, with forest polynomials computed to be $x_1^2+x_2^2+x_1x_2$ and $x_1x_2+x_1x_3+x_2x_3$ respectively. Since $\forestpoly_{\gamma^{-1}(T_0)}=x_1^2x_2+x_1^2x_3$, we retrieve the expansion for $\forestpoly_{T_0}$ in Example~\ref{exa:pT0}.
\end{example}

Lemma~\ref{lem:leaf_indexed} gives a nice interpretation for the sum of coefficients of $\forestpoly_F$, i.e. the specialization $\forestpoly_F|_{x_i=1 \forall i\geq 1}$.
Given the multiplicative property in~\eqref{eq:multiplicative}, it suffices to give the answer in the case of an indexed tree $T$.
We extend an earlier definition of a decreasing labeling mildly. Given $n\geq |T|$, let $\decreasing(T,n)$ denote the set of decreasing labelings of $T$ with distinct integers drawn from $[n]$.

\begin{proposition}
\label{prop:all_1s_tree}
Let $T$ be an indexed tree and let $M\coloneqq\max(\supp(T))$.
\begin{equation}
\label{eq:all_1s_tree}
\forestpoly_T(1,\dots,1)=|\decreasing(T,M)|=\binom{M}{|T|}\decreasing(T).
\end{equation}
\end{proposition}
\begin{proof}
We describe a recurrence for the right-hand side in~\eqref{eq:all_1s_tree} which turns out to be equivalent to~\eqref{eq:leaf_indexed} when specializing $x_i=1$ for all $i\geq 1$.

Any decreasing labeling of $T$ using distinct integers from $[M]$ either uses $1$ or it does not. If the latter, then we are counting decreasing labelings of $T$ with labels drawn from $\{2,\dots,M\}$.
If the former, then the label $1$ must appear at one of the terminal nodes, and we are again left to count decreasing labelings of the remaining tree (trimming a tree leaves a tree) with labels from $\{2,\dots,M\}$.
To conclude, note that decreasing labelings with labels drawn from $\{2,\dots,M\}$ are equinumerous with decreasing labelings with labels drawn from $[M-1]$.
\end{proof}

In the case $T$ is an indexed tree with support $[n-1]$, one can track the bijection between linear extensions of $T$ and monomials in $\forestpoly_T$ at a more granular level.
For $\ell \in \decreasing(T)$, we abuse notation and refer to the permutation in $S_{n-1}$ obtained by reading the labels in inorder by $\ell$.  
\begin{corollary}
\label{cor:sum_over_codes}
	For $T$ an indexed tree with support $[n-1]$ we have
	\[
	\forestpoly_T=\sum_{\ell\in \decreasing(T)}\alpx_{\lcode(\ell^{-1})+1^{n-1}}
	\]
	where $\widetilde{\bfc}\coloneqq \lcode(\ell^{-1})+1^{n-1}$ is obtained by adding $1$ to every entry in the code for $\ell^{-1}$ and $\alpx_{\widetilde{\bfc}}\coloneqq \prod_{1\leq i\leq n-1}x_{\widetilde{c}_i}$.
\end{corollary}

\begin{proof}
The right-hand side obeys the recurrence ~\eqref{eq:leaf_indexed}: note that in this case $\gamma^{-1}(T)$ has $0$ in its support and thus the associated forest polynomial vanishes.
\end{proof}

\begin{example}
Consider the indexed tree $\gamma^{-1}(T_0)$ obtained from $T_0$ in Figure~\ref{fig:forest_0201}. 
Then its support is $[3]$. The two linear extensions on the right in Figure~\ref{fig:bf_into_bs_example} are $231$ and $132$. Their inverse have Lehmer codes $(2,0,0)$ and $(0,1,0)$ respectively. Corollary~\ref{cor:sum_over_codes} then says that
\[
\forestpoly_T=x_1^2x_3+x_1^2x_2.
\]
\end{example}

\section{Forest polynomials and $\qsym_n^+$}
\label{sec:forest_and_qsym}
Treat $F$ as the Hasse diagram of a poset on $\internal(F)$ like mentioned before.
We consider decompositions $\internal(F)=L\sqcup U$ where $L$ is a lower order ideal and $U$ is the upper order ideal determined by $\internal(F)\setminus L$.
Given such a decomposition, we say that a labeling $\kappa:\internal(F)\to \bZ_{>0}$ is \emph{$L$-compatible} if the following conditions hold:
\begin{enumerate}
\item $\kappa|_{U}$ obeys the conditions of being an ordinary $P$-partition, i.e. the labels increase weakly going down  left edges and increase strictly going down right edges.
\item the labels of $\kappa|_L$ decrease strictly going down left edges and decrease weakly going down right edges, and $\kappa(v)>\rho_F(v)$ for every $v\in L$.
\end{enumerate}
In short, $\kappa$ is $L$-compatible if it obeys $P$-partition conditions on $L$ and flagged {\em dual} $P$-partition conditions on $U$. 

For a fixed indexed forest $F$, let $\compl(F)$ denote the set of all ordered pairs $(\kappa,L)$ where $\kappa$ is $L$-compatible and $L$ is a lower order ideal in $F$.
Given $y=(\kappa,L)\in \compl(F)$, define the weight $\wt(y)$ to equal the monomial $\alpx_{\kappa}\coloneqq \displaystyle\prod_{v\in \internal(F)} x_{\kappa(v)}$.
Its \emph{sign} $\mathrm{sign}(y)$ equals $(-1)^{|L|}$.
Since the weight only depends on $\kappa$, we may occasionally write $\wt(\kappa)$ instead of $\wt(y)$.

\begin{example}
Figure~\ref{fig:L_compatible_1} shows an indexed forest $F$ with two different $(\kappa,L)\in \compl(F)$. The highlighted region contains nodes from $L$.  The weight in both cases is $\textcolor{blue}{x_3^2x_8x_{10}}\,\textcolor{red}{x_4x_5x_7x_{10}^2}$ but the signs are opposite, as the $L$ on the right is obtained by removing an element from the $L$ on the left.
\end{example}

\begin{figure}[!ht]
\includegraphics[width=0.75\linewidth]{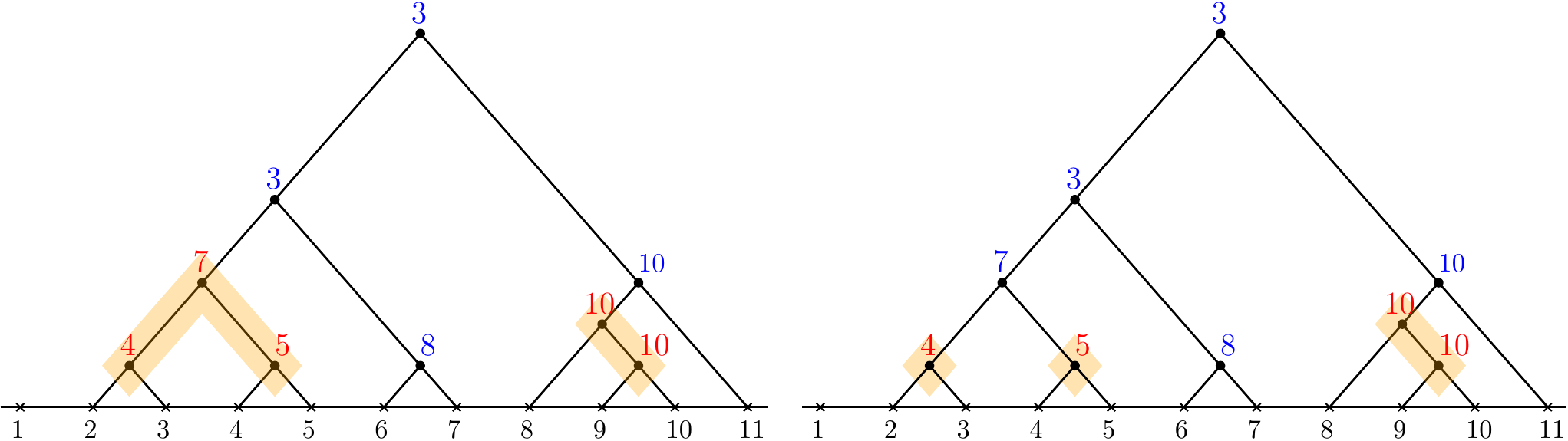}
\caption{Two instances of $L$-compatible labelings with the same underlying indexed forest but differing $L$'s. The weights are equal but the signs are opposite.}
\label{fig:L_compatible_1}
\end{figure}

\begin{theorem}
\label{th:forest_into_dual_forest}
Given an indexed forest $F$, we have
\[
\forestpoly_F=\sum_{y\in \compl(F)} \wt(y)\mathrm{sign}(y)=\sum_{\text{lower ideals } L} (-1)^{|L|}\sum_{\text{$L$-compatible }\kappa}\alpx_{\kappa}.
\]
\end{theorem}
We prove this in Appendix~\ref{app:sign_reversing} by a partial sign-reversing weight-preserving involution on $\compl(F)$.

\subsection{Reduction modulo $\qsym_n^+$}
\label{subsec:red_mod_qsymn}

Note that Theorem~\ref{th:forest_into_dual_forest} expresses the forest polynomial as a signed sum of formal power series.
Looking over the proof should convince the reader that we could have restricted the image of our labelings to $[n]$ for a positive integer $n$ rather than $\bZ_{>0}$.
This amounts to setting $x_i=0$ for all $i>n$ in Theorem~\ref{th:forest_into_dual_forest}.
Since $\forestpoly_F$ only involves the variables $x_1$ through $x_{|\lsupp(F)|}$, for the rest of this section let us pick $n$ to be such that $\lsupp(F)\subseteq [n]$ and restrict our $L$-compatible labeling $\kappa$ to take values in $[n]$.

Let us define the \emph{dual forest polynomial} $\widetilde{\forestpoly}_F$ as:
\begin{align}
\label{eq:dual_forest}
\widetilde{\forestpoly}_F=\sum_{\text{$\internal(F)$-compatible } \kappa}\alpx_{\kappa}.
\end{align}
We then have the following corollary of Theorem~\ref{th:forest_into_dual_forest} that renders reduction modulo $\qsym_n^+$ into a routine task.
Recall the ABB decomposition
\[
\bQ[\alpxn{n}]=\bQ\{\alpx^{\bfc}\suchthat \bfc \in \abb{n}\}\oplus \qsym_n^+.
\]
\begin{corollary}
\label{cor:forest_and_qsym}
For $F\in \IF$ satisfying $\lsupp(F)\subseteq [n]$,  we have
\[
\forestpoly_F= (-1)^{|F|}\widetilde{\forestpoly}_F \mod \qsym_n^+.
\]
As a consequence,
\[\begin{cases}
\forestpoly_F\in\qsym_n^+\text{ if }\supp(F)\nsubseteq [n-1],\\
\forestpoly_F\in \bQ\{\alpx^{\bfc}\suchthat \bfc \in \abb{n}\}\text{ otherwise}.
\end{cases}\]
\end{corollary}
\begin{proof}
For any lower order ideal $L$ whose underlying set is not all of $\internal(F)$, we have that
\begin{align}
\label{eq:sum_in_question}
\sum_{\text{$L$-compatible }\kappa }\alpx_{\kappa} \in \qsym_n^+.
\end{align}
This is because the $P$-partition generating polynomial $K_{U}(x_1,\dots,x_n)$, which is a quasisymmetric polynomial, is a factor of the sum in~\eqref{eq:sum_in_question}.
Thus, modulo $\qsym_n^+$ we only need to concern ourselves with the case $L=\internal(F)$. By definition, the sum in that case is $\widetilde{\forestpoly}_F$. The first claim follows.

One half of the second statement follows because $\widetilde{\forestpoly}_F$ is identically $0$ as  there are no $\internal(F)$-compatible fillings if $\supp(F)\nsubseteq [n-1]$. The other half is by Proposition~\ref{prop:forests_as_basis}.
\end{proof}

\begin{remark}
Suppose $F$ is an indexed forest with $\supp(F)\subseteq [n-1]$.
Let $\widetilde{F}$ denote the indexed forest obtained by reflecting $F$ in the interval $[n-1]$, i.e. $i\in \supp(F)$ if and only if $n-i\in \supp(\widetilde{F})$.
It is easily checked that $
(-1)^{|F|}\widetilde{\forestpoly}_F=\forestpoly_{\widetilde{F}}(-x_n,\dots,-x_1).$
Thus, in light of Corollary~\ref{cor:forest_and_qsym}, we have  
\[
\forestpoly_F(x_1,\dots,x_n)=\forestpoly_{\widetilde{F}}(-x_n,\dots,-x_1) \mod \qsym_n^+.\]
To emphasize the parallel to Schubert polynomials, recall that for $w\in S_{n}$ we have 
\[
\schub{w}(x_1,\dots,x_n)=\schub{w_0ww_0}(-x_n,\dots,-x_1) \mod \Sym_n^+.
\]
\end{remark}

\section{A correspondence between words and forests}
\label{sec:wf}

\subsection{The correspondence $\wfcorr$ via an insertion algorithm}
\label{subsec:wf}

Let $W$ be a word in $\injwords{\Zaug}$, so $W=W_1\cdots W_r$ where the $W_i$ are pairwise distinct letters in $\Zaug$. 
We will define by induction two labelings $\psymb(W)\in \lbs(F)$ and $\qsymb(W)\in \decreasing(F)$ of the same indexed forest $F=F(W)$. 
The reader may find it helpful to refer to Figure~\ref{fig:correspondence} which illustrates the algorithm described next.

\begin{enumerate}[align=parleft,left=0pt,label=(\roman*)]
\item
If $W$ is the empty word $\epsilon$, then  $F$ is the empty tree.
 
\item
If $W=W'a$, then we assume by induction on $|W|$ that we have already constructed $\psymb(W')\in \lbs(F')$ and $\qsymb(W')\in\decreasing(F')$ for $F'=F(W')$ an indexed forest with support $S'$.  
 Let $S'$ be decomposed into maximal intervals  $I_1,\ldots,I_k$ from left to right. 
 For any $j\in [k]$, let $T_j$ be the indexed tree with support $I_j$, $u_j$ be its root, and $a_j\in\Zaug$ be the root label in $\psymb(W')$.
We define $F(W)$ by inserting a new node $u$ to $F'$. 
\begin{itemize}
\item
 Suppose that $i=\mathrm{val}(a)\notin S'$. 
 Then we can uniquely define $F$ by imposing $\can_F(u)=i$. 
 Explicitly, 
 \begin{enumerate}
  \item If $i-1\in S'$, say $i-1\in I_j$, then let the left child of $u$ be $u_j$; otherwise it has no left child.
   \item If $i+1\in S'$, say $i+1\in I_j$, then let the right child of $u$ be $u_j$; otherwise it has no right child.
    \item If both $i-1,i+1\notin S'$, then we have a new tree formed with root $u$, and we assign to it the support $\{i\}$.
  \end{enumerate}
    
   \item 
    Now suppose $i=\mathrm{val}(a)\in I_j$ for some $j\in [k]$. 
    \begin{enumerate}
      \item If $a>a_j$, then we let $u_j$ be the left child of $u$. 
    If $\max(I_j)+2\in S$ (it is thus necessarily equal to $\min(I_{j+1})$), then the right child of $u$ is $u_{j+1}$; otherwise $u$ has no right child.
    \item If $a<a_j$, then we let $u_j$ be the right child of $u$. 
    If $\min(I_j)-2\in S$ (it is thus necessarily equal to $\max(I_{j-1})$), then the left child of $u$ is $u_{j-1}$; otherwise $u$ has no left child.
    \end{enumerate}
       \end{itemize}
    \end{enumerate}
    
    This determines $F(W)\in \IF$. 
    Now $\psymb(W)$ is obtained from $\psymb(W')$ by labeling the new node $u$ by $a$, while $\qsymb(W)$ is obtained from $\qsymb(W')$ by labeling the node $u$ by $|W|$. We have thus defined $\wfcorr(W)\coloneqq (\psymb(W),\qsymb(W))$ and the common shape $F(W)$.

\begin{figure}[!ht]
\includegraphics[width=\textwidth]{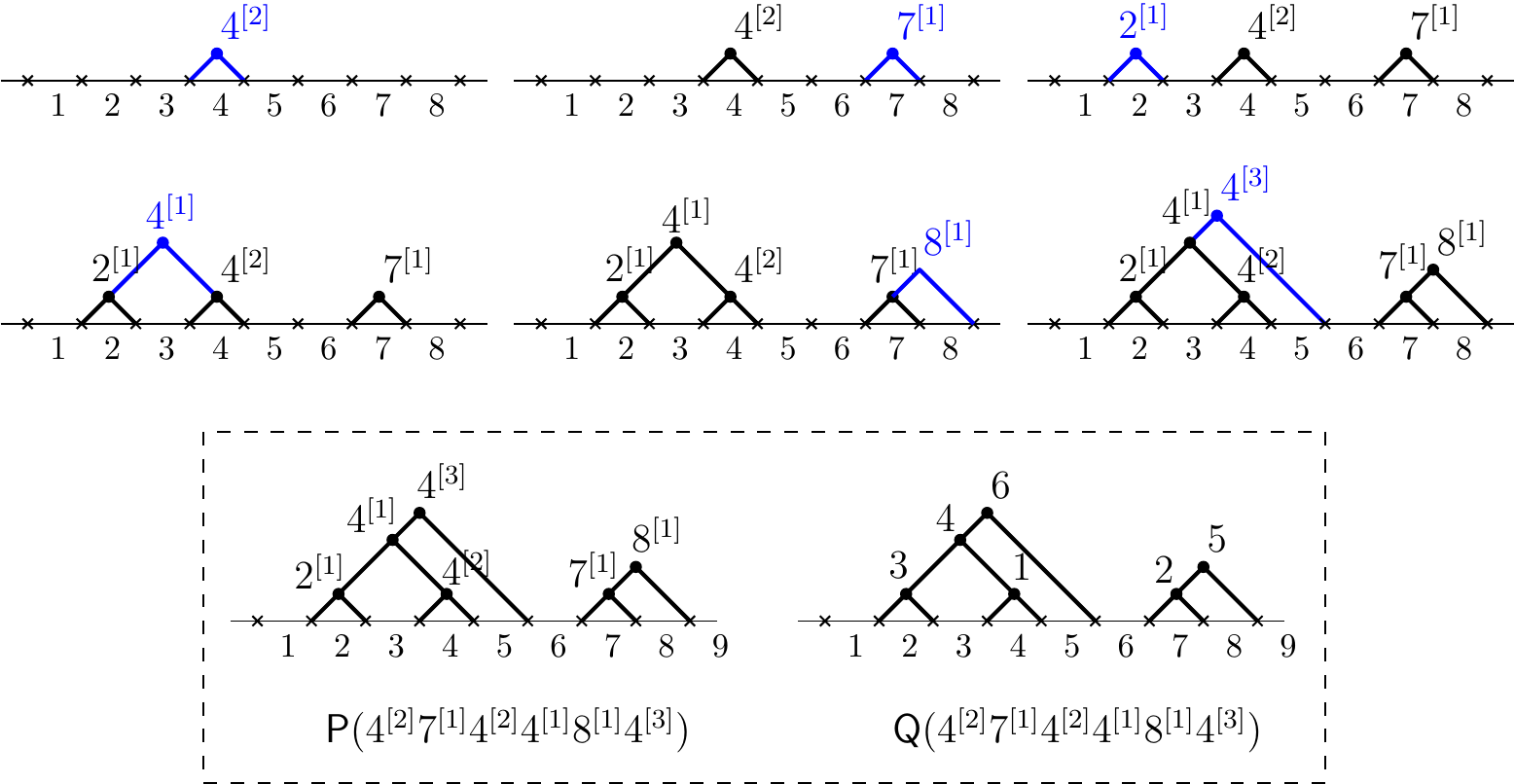}
\caption{\label{fig:correspondence} The correspondence $\wfcorr$ applied to  $W=\lf{4}{2}\lf{7}{1}\lf{4}{2}\lf{4}{1}\lf{8}{1}\lf{4}{3}$. The first two rows show the letter-by-letter creation of the $\psymb$-symbol.}
\end{figure}

\begin{theorem}
The correspondence $\wfcorr$ is one-to-one between:
\begin{enumerate}[label=\emph{(\arabic*)}] 
\item Injective words $W$ with letters in $\Zaug$, and
\item  Pairs $(\psymb,\qsymb)$ such that $\psymb\in \lbs(F)$ and $\qsymb\in \decreasing(F)$ for a common $F\in\IF$. 
\end{enumerate} 
\end{theorem}

\begin{proof}
Let $\wfcorr(W)=(\psymb(W),\qsymb(W))$.
The procedure is well defined by immediate induction: $\psymb(W)$ and $\qsymb(W)$ have the same underlying $F\in\IF$, and $\psymb(W)$ belongs to $\lbs(F)$ while $\qsymb(F)$ is in $\decreasing(F)$ by simply inspecting the construction.

Given $(\psymb,\qsymb)$ with common shape $F\in \IF$, define $W$ to be the word obtained by reading the labels in $P$ in the order given by the labels $1,2,\ldots,$ in $\qsymb$.
 Define $\Gamma$ by $W=\Gamma(\psymb,\qsymb)$.
  It is clear by construction that $\Gamma\circ \wfcorr$ is the identity on injective words.
   This shows already that $\wfcorr$ is injective. 

Now if $W=\Gamma(\psymb,\qsymb)$, we need to show that $\wfcorr(W)=(\psymb,\qsymb)$. 
This is true if $F$ is the empty forest, so assume $F$ is nonempty. 
Let $u$ be the vertex with the maximal label $|F|$ in $\qsymb$. 
Let $a$ be the label of $u$ in $\psymb$. 
By the definition of $\Gamma$ we have $W=W'a$ for a word $W'$. 
Let $\psymb'$ and $\qsymb'$ be the labelings obtained by restricting $\psymb$ and $\qsymb$ respectively to the shape $F'=F\setminus\{u\}$. 
Again by the definition of $\Gamma$ we have $W'=\Gamma(\psymb',\qsymb')$, so by induction we may assume $\wfcorr(W')=(\psymb',\qsymb')$. 

Now we want to insert $a$ in $\psymb'$. Following the definition of $\wfcorr$, it is a routine verification that this insertion results in $\psymb$, thereby completing the proof. 
\end{proof}

\begin{remark}
If $W$ is a permutation in $S_n$ in one-line notation, then the resulting shape $F(W)$ is a tree $T(W)$ supported on $\{1,\ldots,n\}$ and  $\psymb(W)$ is the canonical labeling. All the information is thus contained in the decreasing binary tree $\qsymb(W)$. The correspondence $W\to \qsymb(W)$ is easily seen to be the standard bijection between permutations and decreasing trees. 

This particular case can also be seen as the special case of the \emph{Sylvester correspondence} \cite{HNT05}. However while the extension of the latter to general words can be done by standardization, this is clearly not the case for our procedure.
\end{remark}

\subsection{$\wf$-Parking procedure}
\label{subsec:wf_parking}

We now employ the correspondence $\wfcorr$ to define a new parking procedure for words in $\injwords{\Zaug}$ with the underlying idea being to keep track of the support rather than the whole indexed forest.
We see our procedure as a function $\wf$ associating to $W\in\injwords{\Zaug}$ a finite subset $S=\wf(W)$ of $\bZ$ with size $|S|=|W|$. 
A restrictive version of $\Omega$ was already stated in the introduction.

The precise description is as follows. 
Write $W=W'a$, and assume $\wf(W')$ is a subset with interval decomposition $(I_j)_j$. 
\begin{enumerate}[align=parleft,left=0pt,label=\text{(\arabic*)}]
 \item If $\mathrm{val}(a)\notin\wf(W')$, then simply set $\wf(W)=\wf(W')\sqcup\{\mathrm{val}(a)\}.$ 
\item If $\mathrm{val}(a)\in I_j$ for some $j$,  consider the largest $i\in \{1,\ldots,|W'|\}$ such that $W_i\in I_j$. Then set 
\begin{align*}
\wf(W)=\left\lbrace \begin{array}{ll}\wf(W')\sqcup\{\min(I_j)-1\} & a<W_i
\\ 
\wf(W')\sqcup\{\max(I_j)+1\} & a>W_i.\end{array}\right.
\end{align*}
\end{enumerate}

Comparison with the definition of $\wfcorr$ gives:
\begin{proposition}
\label{prop:correspondence_and_parking}
For any word $W\in\injwords{\Zaug}$, one has $\supp(\wfcorr(W))=\wf(W)$. In particular 
\[
\text{$W$ is $\wf$-parking $\Longleftrightarrow$ $F(W)$ has support $\{1,\ldots,|W|\}$.}
\]
\end{proposition}

\begin{remark}
The relation between $\wf$-parking and $\wfcorr$-correspondence is an instance of a more general phenomenon detailed by the first author in~\cite{NadParking}: any so-called ``bilateral'' parking procedure naturally gives rise to a similar correspondence with pairs of labeled indexed forests.
\end{remark}

\subsection{The $\wf$-equivalence}
\label{subsec:wf_equivalence}

Our main use for the $\wfcorr$-correspondence is to introduce the following~notion.

\begin{definition}
The equivalence relation $\wfequiv$ on $\injwords{\Zaug}$ is defined by $W\wfequiv W'$ if $\psymb(W)=\psymb(W')$.
\end{definition}

The equivalence classes under $\wfequiv$ are thus naturally indexed by $P\in \lbs$. Explicitly,  the class of $P$ is $\mathcal{C}_P=\{W\in\injwords{\Zaug}\suchthat \psymb(W)=P\}$. We will also write $\mathcal{C}(W)$ for the $\wfequiv$-class of $W$.

Let $P\in \lbs(F)$ for some $F\in\IF$. Since $\wfcorr$ is bijective we have 
\begin{equation}
\label{eq:P_class}
\mathcal{C}_P=\{W\in\injwords{\Zaug}\suchthat \wfcorr(W)=(P,Q)\text{ for some }Q\in \decreasing(F)\}.
\end{equation}
Any two decreasing forests of the same shape are connected by a sequence of swaps of labels $i,i+1$ for some $i$, where the vertices labeled $i,i+1$ are not descendants of one another. So if $Q,Q'$ are two decreasing forests related by such a swap, we write $W\simwf W'$ if $\wfcorr(W)=(P,Q)$ and $\wfcorr(W')=(P,Q')$. 
By \eqref{eq:P_class}, we thus have that $\simwf$ generates $\wfequiv$. 

We now translate this perspective of performing swaps on trees to performing a commutation move on words.
Given $W,W'\in\injwords{\Zaug}$, we have $W\simwf W'$ if and only if $W=UabV$, $W'=UbaV$ for some $a,b\in\Zaug$ such that when inserting $a$ followed by $b$ in $\psymb(U)$, the node labeled $a$ is not a child of the node labeled $b$. 
For this, one needs to look at labels of the root nodes in the trees constituting $\psymb(U)$ as well as the support of $\psymb(U)$. 

As a bookkeeping aid, we introduce an alphabet $\Zleaves=\{\underline{i}\suchthat i\in\bZ\}$ and extend the total order on $\Zaug$ to $\Zleaves\sqcup \Zaug$ by $\underline{i}<\lf{i}{j}<\underline{i+1}$ for any $i,j$. 
Pictorially, while $\bZ$ (and by extension $\Zaug$ via $\mathrm{val}:\Zaug\to\bZ$) in indexed forests was naturally associated with unit intervals on the integer line, we now think of $\Zleaves$ as points in the integer line.
 The reader may benefit from interpreting the crosses, i.e. leaves, in all our figures as elements of $\Zleaves$, with $\underline{i}$ and $\underline{i+1}$ representing the left and right endpoints respectively of the unit segment labeled by $i\in \bZ$.

\begin{definition}
\label{def:rl}
For $P\in \lbs$,  the {\em rootlist} $\rootlist(P)$ of $P$ is defined as the totally ordered subset of $\Zleaves\sqcup \Zaug$ containing all root labels of $P$, as well as all $\underline{i}$ whenever $i-1,i\notin \supp(P)$. 
We define by extension $\rootlist(W)=\rootlist(\psymb(W))$.
\end{definition}
Note that $\rootlist(P)$ is infinite, containing $\underline{i}$ and $\underline{-i}$ for  all $i$ large enough. 
 Going back to Figure~\ref{fig:correspondence}, the rootlist of the labeled indexed forest $P$ in the middle of the second row has elements $\{\cdots <\underline{0}<\underline{1}<\lf{4}{1}<\underline{6}<\lf{8}{1}<\underline{10}<\cdots\}$.
 
 A simple inspection of the insertion procedure $\wfcorr$ yields the following lemma that shows how the rootlist changes upon inserting a single letter.
 \begin{lemma} 
\label{lem:rootlist_computation} Consider a word $W$ and a letter $a$ such that $Wa\in\injwords{\Zaug}$. 
Let $r_1,r_2$ be \emph{the} two elements in $\rootlist(W)$ that are consecutive for the total order on $\rootlist(W)$ and satisfy $r_1<a<r_2$. Then 
\begin{align*}
\rootlist(Wa)=\rootlist(W)\setminus \{r_1,r_2\}\sqcup\{a\}.
\end{align*} 
 \end{lemma}
 As an example, consider $W=\lf{4}{2}\lf{7}{1}\lf{4}{2}\lf{4}{1}\lf{8}{1}$ and $a=\lf{4}{3}$. 
By referring to the second row in Figure~\ref{fig:correspondence} we see that $r_1=\lf{4}{1}$ and $r_2=\underline{6}$. 
On inserting $a$ as per $\wfcorr$, both $r_1$ and $r_2$ get removed from the rootlist, while $a$ gets added.
 
 Equipped with this elementary insight into the rootlist, we recast $\simwf$ as a commutation move. 
 The preceding discussion and lemma imply the next two results.
\begin{proposition}
\label{prop:rootlist_separation}
Given $W,W'$, we have $W\simwf W'$ if and only if $W=UabV$, $W'=UbaV$ for $a,b\in\Zaug$ and $a,b$ are separated by at least two elements of $\rootlist(U)$, i.e. there exist $r_1<r_2\in \rootlist(U)$ such that $a<r_1<r_2<b$ or $b<r_1<r_2<a$.
\end{proposition}

\begin{lemma}
\label{lem:anti_monoid_congruence}
If $UabV\simwf UbaV$, then  $U'abV'\simwf U'baV'$ for any subword $U'$ of $U$ and $V'$ of $V$.
\end{lemma}
Indeed by Proposition~\ref{prop:rootlist_separation} the suffix $V$ plays no role, and using that proposition together with Lemma~\ref{lem:rootlist_computation} shows the result.

We are now in a position to give the enumerative consequences of the previous results.

\subsection{From slide polynomials to forest polynomials}
Recall the polynomials $\forestpoly(P)$ from Section~\ref{sec:forest_polynomials}. By~\eqref{eq:P_class}, words in $\mathcal{C}_P$ correspond bijectively to decreasing forests  with the same underlying indexed forest. 
From~\ref{eq:forests_into_slides_generalized} we know that
\[
\forestpoly(P)=\sum_{W\in\mathcal{C}_P}\slide(W),
\]
which in turn leads us:\begin{proposition}
\label{prop:from_slides_to_forests}
If $A\subset \injwords{\Zaug}$ is a finite subset of words in $\Zaug$ that is stable under $\wfequiv$, then
\[
\sum_{W\in A} \slide(W)=\sum_{P\in \psymb(A)}\forestpoly(P),
\]
where $\psymb(A)$ is the set of $\psymb{W}$ for $W\in A$. 
\end{proposition}

We now have our main theorem:
\begin{theorem}
\label{th:slide_to_parking}
Let $n\geq 1$. Let $A\subset \injwords{\Zaug}$ be composed of words of length $n-1$, with all letters having values in $[n]$, and stable under $\wfequiv$.
Then 
\[
\ds{\sum_{W\in A} \slide(W)}{n} =|\text{$\wf$-parking words in $A$}|.
\]
\end{theorem}

\begin{proof}
This is simply a matter of combining the previous results. 
In the following we write for short $\supp(P)\coloneqq \supp(F)$ when $P\in\lbs(F)$.

By Proposition~\ref{prop:from_slides_to_forests} we have that $\sum_{W\in A} \slide(W)$ is equal to $\sum_{P\in \psymb(A)}\forestpoly(P)$. 
Since letters have values in $[n]$, the underlying indexed forests $F$ in this sum satisfy $\lsupp(F)\subseteq [n]$. 
Since $\forestpoly(P)=\forestpoly_F$ when $P\in\lbs(F)$ by Proposition~\ref{prop:betaP_is_betaF}, we can apply Corollary~\ref{cor:forest_and_qsym}: $\forestpoly(P)$ is in $\qsym_n^+$ unless $\supp(P)\subseteq [n-1]$, and in that case $\forestpoly(P)$ lies in $\bQ\{\alpx^{\bfc}\suchthat \bfc \in \abb{n}\}$. 
The latter occurs  when $\supp(P)=[n-1]$. Indeed, since all words have length $n-1$, the forests have size $n-1$.

It follows that 
\[\sum_{W\in A} \slide(W)=S_A\coloneqq \sum_{P\in \mathcal{S}(A)}\forestpoly(P)\text{ modulo } \qsym_n^+,\] where $\mathcal{S}(A)$ consists of those $P\in \psymb(A)$ such that $\supp(P)=[n-1]$.
By \cite[Theorem 1.3]{DS}, recalled in Equation~\eqref{eq:ds_and_qsym}, we know that 
\[\ds{\sum_{W\in A} \slide(W)}{n}=S_A(1,\ldots,1).\]
 By applying Proposition~\ref{prop:all_1s_tree}, we obtain
\begin{align*}
\ds{\sum_{W\in A}\slide(W)}{n}&=\sum_{P\in \mathcal{S}(A)}\forestpoly(P)(1,\ldots,1)=\sum_{P\in \mathcal{S}(A)}|\decreasing(P)|\\
&=\sum_{P\in \mathcal{S}(A)}|\mathcal{C}_P|.
\end{align*}
By the definition of $\mathcal{S}(A)$, this sum is precisely the number of words $W\in A$ such that $\psymb(W)$ has support $[n-1]$. By Proposition~\ref{prop:correspondence_and_parking}, this is the number of $\wf$-parking words in $A$.
\end{proof}

\section{Applications}
\label{sec:applications}

\subsection{Combinatorial interpretation of $a_w$}
\label{subsec:combinatorial_interpretation}

We can finally prove Theorem~\ref{th:aw_as_parking}.
Given that our $\Omega$-parking procedure takes words in 
$\injwords{\Zaug}$ as input, to apply it to reduced words (which are words in $\bZ$) we first transform them into words in $\injwords{\Zaug}$.
This is done by replacing all instance of the letter $i\in \bZ$ by $\lf{i}{1}$, $\lf{i}{2},\dots$ from left to right.
So for instance the reduced word $1343$ for $w=21543$ becomes $\lf{1}{1}\lf{3}{1}\lf{4}{1}\lf{3}{2}$.

\begin{remark}
The preceding transformation breaks ties between equal letters in words in $\bZ$ in a specific manner. 
It is worth emphasizing that when dealing with reduced words, how we relabel the letters with equal valuest plays no role  since we never get to compare these exponents.
\end{remark}

\begin{theorem}
\label{th:the_theorem}
Let $w\in S_n$ have length $n-1$. Then $a_w$ is equal to the number of reduced words of $w^{-1}$ that are $\wf$-parking words.
\end{theorem}
\begin{proof}
Recall from~\eqref{eq:schub_def} that \[\schub{w} =\sum_{\mathbf{i}\in \red(w^{-1})}\slide(\mathbf{i}).\] 
 Now any set of words stable under commutations $ab \leftrightarrow ba$ for $|\mathrm{val}(a)-\mathrm{val}(b)|>1$ is automatically stable under $\wfequiv$ thanks to Proposition~\ref{prop:rootlist_separation}. In particular, $\red(w^{-1})$ is stable under $\wfequiv$. Since $w\in S_n$ has length $n-1$, all words in $\red(w^{-1})$ have length $n-1$ and letters in $[n]$ (in fact, in  $[n-1]$). We can thus apply Theorem~\ref{th:slide_to_parking} to get the result.
\end{proof}

An example is given in the introduction. We will come back to this interpretation in Section~\ref{sec:the_rest} so as to relate it to several results from ~\cite{NT20}.

\subsection{Multiplication of forest polynomials}
\label{subsec:mult}

Given nonempty words $U=u_1\cdots u_a$ and $V=v_1\cdots v_b$, 
let $W=w_1\cdots w_{a+b}\coloneqq u_1\cdots u_a v_1\dots v_b$ be their concatenation.
The multiset of shuffles $\mathsf{Sh}(U,V)$ is defined as 
\[
\{
(w_{\sigma(1)},\dots, w_{\sigma(a+b)})\suchthat \sigma\in S_{a+b}, \sigma^{-1}(1)<\cdots<\sigma^{-1}(a) \text{ and }\sigma^{-1}(a+1)<\cdots < \sigma^{-1}(a+b)\}.
\]
Given sets of words $\mathcal{C}$ and $\mathcal{D}$, we extend this definition and let $\mathsf{Sh}(\mathcal{C},\mathcal{D})$ to be the (multiset) union of all $\mathsf{Sh}(U,V)$ where $U\in \mathcal{C}$, $V\in \mathcal{D}$.

\begin{theorem}
\label{th:mult_forest_poly}
Let $F_1,F_2\in \IF$. Consider the basis expansion
\[\forestpoly_{F_1}\forestpoly_{F_2}=\sum_{G\in \IF} b_{F_1,F_2}^G\forestpoly_G.\]
Then the coefficients $b_{F_1,F_2}^G$ are in $\bZ_{\geq 0}$.
\end{theorem}

\begin{proof}
We write $\forestpoly_{F_i}=\forestpoly(P_i)$ for $i=1,2$, for some $P_i\in\lbs(F_i)$. We assume without loss of generality that the node labels in $P_1$ and $P_2$ are disjoint. 
Indeed, changing all node labels $\lf{i}{j}$ in $P_2$ to $\lf{i}{j+N}$ preserves the property of being in $\lbs(F_2)$.
So, for $N$ large enough we have disjoint sets of labels.

Now $\forestpoly(P_i)$ is the sum of all $\slide(W)$ for $W\in \mathcal{C}_{P_i}$. 
Furthermore the product $\slide(W)\slide(W')$ (where $WW'\in\injwords{\Zaug}$) is the sum of $\slide(X)$ for all words $X$ in the shuffle $\mathsf{Sh}(W,W')$; see~\cite{NT23}. Therefore 
\[\forestpoly_{F_1}\forestpoly_{F_2}=\sum_{W\in\mathsf{Sh}(\mathcal{C}_{P_1},\mathcal{C}_{P_2})}\slide(W).
\]

By Proposition~\ref{prop:from_slides_to_forests} one needs to show that $A=\mathsf{Sh}(\mathcal{C}_{P_1},\mathcal{C}_{P_2})$ is stable under $\wfequiv$. 
We only need to show that $A$ is stable under $\simwf$. 

Let $W\in\mathsf{Sh}(W_1,W_2)$ for $W_i\in\mathcal{C}_{P_1}$, $i=1,2$, and assume there exists a factorization $W=UabV$ such that $W\simwf UbaV$. 
We need to show $UbaV\in \mathsf{Sh}(\mathcal{C}_{P_1},\mathcal{C}_{P_2})$. 
If $a$ and $b$ are letters in $W_1,W_2$ respectively, or in $W_2,W_1$ then $UbaV\in\mathsf{Sh}(W_1,W_2)$ by definition of shuffles. 

Otherwise $a$ and $b$ are letters in $W_i$ for some $i\in\{1,2\}$. 
We may assume $i=1$ without loss of generality. We have then a factorization $W_1=U_1abV_1$. By Lemma~\ref{lem:anti_monoid_congruence}, we have $W_1\simwf W'_1$ where $W'_1=U_1baV_1$. In particular $W'_1\in\mathcal{C}_{P_1}$, and thus $UbaV\in\mathsf{Sh}(W'_1,W_2)$ as wanted.
\end{proof}

\subsection{Multivariate mixed Eulerian numbers}
\label{subsec:mme}

Fix $n$ a positive integer.
Given $\bfc=(c_1,\dots,c_{n-1})\in \nvectsn{n-1}$ of weight $n-1$ we define 
\[
y_{\bfc}\coloneqq x_1^{c_1}(x_1+x_2)^{c_2}\cdots (x_1+\cdots+x_{n-1})^{c_{n-1}}.
\]
This given, define the \emph{multivariate mixed Eulerian number} $A_{\bfc}(x_1,\dots,x_{n-1})$ as follows
\begin{align}
\label{eq:def_multivariate_me}
	A_{\bfc}(x_1,\dots,x_{n-1})= y_{\bfc} \mod \qsym_n^+.
\end{align}
Recall we take representatives modulo $\qsym_n^+$ in $\bQ\{\alpx^{\bfc}\suchthat \bfc\in \abb{n}\}$.

Our choice of nomenclature is explained by the fact that setting $x_i=1$ for all $1\leq i\leq n-1$ recovers Postnikov's mixed Eulerian numbers \cite[Section 16]{Pos09}. Furthermore, setting $x_i=q^{i-1}$ for $1\leq i\leq n-1$ recovers the remixed Eulerian numbers introduced by the authors \cite{NT21,NTremixed}.

Consider the alphabet $\bZ_{\bfc}=\displaystyle\bigsqcup_{i=1}^{n-1} \{\lf{i}{j}\suchthat 1\leq j\leq c_i\}$.
Let $\injective_{\bfc}$ denote the set of injective words of length $n-1$ in the alphabet $\bZ_{\bfc}$. It is clear that $|\injective_{\bfc}|=(n-1)!$.

\begin{proposition}
The mixed Eulerian number $A_{\bfc}(1^{n-1})$ is equal to the number of words in $\injective_{\bfc}$ that are $\wf$-parking words.
\end{proposition}

\begin{proof}
 Write $y_{\bfc}$ in the form $\displaystyle\prod_{a\in \bZ_{\bfc}} \slide(a)$. 
By the multiplication rule for slide polynomials recalled in the proof of Theorem~\ref{th:mult_forest_poly}, we have
\[
y_c=\sum_{W\in\injective_{\bfc}} \slide(W).
\]
Since $\injective_{\bfc}$ is closed under $\wfequiv$,  the answer follows from Theorem~\ref{th:slide_to_parking}.
\end{proof}

From the proof of Theorem~\ref{th:slide_to_parking}, we in fact get a combinatorial expansion of $A_c(x_1,\dots,x_{n-1})$: namely, it is the sum of $\forestpoly(P)$ for all LBS trees $P$  with support $\{1,\ldots,n-1\}$ and with nodes labeled by $\bZ_{\bfc}$. For instance, if $c=(0,2,0,2)$ so that $\bZ_{\bfc}=\{\lf{2}{1},\lf{2}{2},\lf{4}{1},\lf{4}{2}\}$, there are five such LBS trees depicted in Figure~\ref{fig:ac_0202_bis}, and the resulting polynomial is
\[
A_{\bfc}(x_1,\dots,x_4)=x_1x_2x_3x_4+2x_1x_2^2(x_3+x_4)+x_1^2x_2(x_2+x_3+x_4)+x_1x_2^3.
\]

\begin{figure}[!ht]
\centering
	\includegraphics[width=\textwidth]{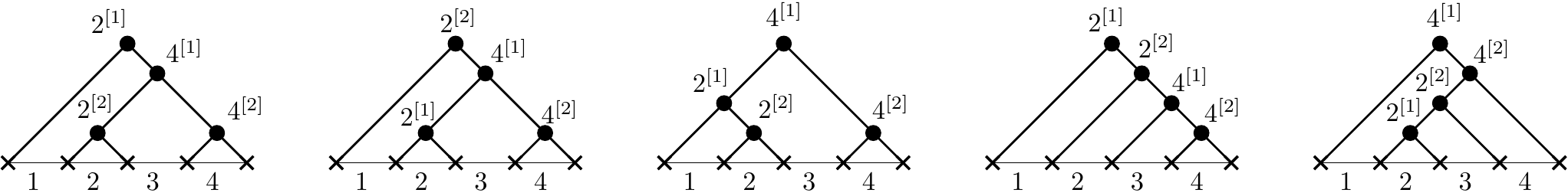}
\caption{All LBS tree with labels $\bZ_{\bfc}=\{\lf{2}{1},\lf{2}{2},\lf{4}{1},\lf{4}{2}\}$}
\label{fig:ac_0202_bis}
\end{figure}

These trees can be connected to the ones in~\cite[Section 7.1]{NTremixed}.  
A bilabeled tree of content $\bfc$ is a triplet $(T,\decreasing,lr)$ with $T$ a binary tree of size $n-1$, $\decreasing$ a decreasing labeling of nodes, and $lr$ a labeling defined on nodes and leaves of $T$ with the following conditions: $lr$ is a bijection to $\{1,\ldots,2n-1\}$, satisfying the local binary search condition (label of any node is strictly larger than the label on its left child and strictly smaller than the label on its right child) and the leaf labels $1=\ell_1<\cdots<\ell_n=2n-1$ from left to right  are determined by $c_i=\ell_{i+1}-\ell_i-1$. That is, $\ell_i=i+c_1+\cdots+c_{i-1}$ for $1\leq i\leq n$. 
We will consider only the labeled trees $(T,lr)$, and illustrate those with content $(0,2,0,2)$ in Figure~\ref{fig:ac_0202}.

\begin{proposition} Pairs $(T,lr)$ of content $\bfc$ are in bijection with LBS trees $P$  with support $\{1,\ldots,n-1\}$ and nodes labeled by $\bZ_{\bfc}$.
\end{proposition}

\begin{proof}
Let $\ell_1<\ldots<\ell_n$ be the leaf labels given by $\bfc$ as above. There is a unique order preserving bijection between $[2n-1]\setminus \{\ell_1,\dots,\ell_n\}$ and $\bZ_{\bfc}$, as these are both totally ordered sets of size $n-1$.

Now given a pair $(T,lr)$ relabel its nodes using the bijection above resulting in an labeling $\kappa$ with image $\bZ_{\bfc}$. 
By forgetting the leaf labels and considering $T$ as an indexed tree of support $[n-1]$, we have $P=(T,\kappa)$ as the desired LBS tree. It is checked that this is the desired bijection. 
\end{proof}
In the case $\bfc=(0,2,0,2)$ the bijection is illustrated in Figures~\ref{fig:ac_0202_bis} and~\ref{fig:ac_0202} with the trees in the same relative position  related via the bijection.

\begin{figure}[!ht]
\centering
	\includegraphics[width=\textwidth]{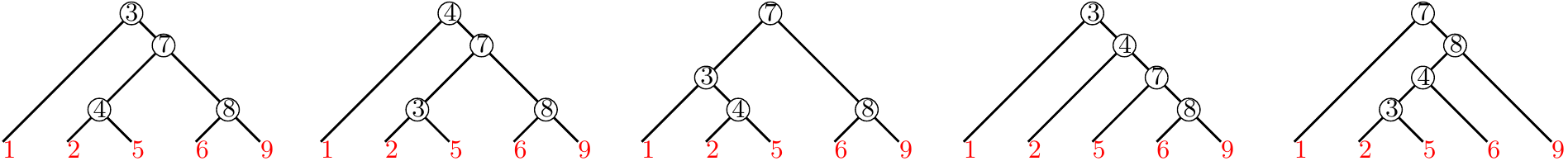}
\caption{
\label{fig:ac_0202} All labeled trees $(T,lr)$ with content $(0,2,0,2)$}
\end{figure}

\section{Further remarks}
\label{sec:the_rest}

\subsection{Back quasisymmetric functions}
Taking a leaf out of \cite{Lam18}, the authors studied we studied the space $\backquasi$ of back stable quasisymmetric functions in \cite{NT23}.
By allowing our labelings $\kappa$ in Definition~\ref{def:forest} to take values in $\bZ$ rather than $\bZ_{\geq 1}$, we obtain a power series $\backforest_F\in \backquasi$ for $F\in\IF$ that we naturally call the \emph{back stable forest polynomial}.
In particular note that $\backforest_F$  is never $0$.

It was shown in \cite{NT23} that the back stable slide polynomials are a basis for $\backquasi$. The slide expansion for forest polynomials in Proposition~\ref{prop:forests_into_slides} permits the obvious back stable analogue. With a little effort one can show that back stable forest polynomials also give a basis for $\backquasi$.
Our Theorem~\ref{th:main_3} from the introduction now translates into the statement that back stable Schubert polynomials expand nonnegatively into back stable forest polynomials, with the underlying combinatorics, relying on the $\wfcorr$-correspondence, entirely unchanged.

The space $\backquasi$ comes equipped with several maps that allow one to conveniently recover the non-back stable story.
Recall the shift map $\gamma$ from Section~\ref{subsec:recurrence} that shifts the support of an indexed forest $F$ one unit to the right, i.e. adds $1$ to each element of $\supp(F)$. Let us abuse notation and refer to the map on the polynomial ring $\bQ[\alpx]$ sending $x_i\mapsto x_{i+1}$ for all $i\in \bZ$ by $\gamma$ as well.
It is then the case that
\begin{align}
\gamma(\backforest_F)=\backforest_{\gamma(F)}.
\end{align}
Under the specialization $x_i=0$ for $i\leq 0$, we see that $\backforest_F$ becomes $\forestpoly_F$ (which is nonzero if and only $F\in \IF_+$). In particular note that the back stable series $\backforest_F$ does contain all information about $\forestpoly_{\gamma^{i}(F)}$ for all $i\in \bZ$.

\subsection{Combinatorics of the numbers $a_w$}

In \cite{NT20}, we obtained several results on the numbers $a_w$ using a different approach. It is interesting to revisit them using mainly the combinatorial interpretation from Theorem~\ref{th:aw_as_parking}. Let us write $S'_n$ for the set of permutations of $S_n$ of length $n-1$.

\begin{enumerate}
[align=parleft,left=0pt,label=(\alph*)]
\item Suppose $w\in S'_n$ is \emph{$m$-Grassmannian}, i.e. its unique descent is in position $m\geq 1$. 
Let $\lambda(w)\coloneqq (\lambda_1,\dots,\lambda_k)$ be the partition of $n-1$ determined from $\lcode(w)$ by omitting the $0$s.
Then $a_w$ counts the number of standard Young tableaux of shape $\lambda(w)$ with $m-1$ descents \cite[Theorem 6.4.2]{NT20}.

This can now be understood bijectively with our interpretation, let us sketch the proof. Reduced words of $w^{-1}$ are  in bijection with standard Young tableaux of shape $\lambda(w)$: Explicitly, $\mathbf{i}=i_1\cdots i_\ell\in\red(w^{-1})$ determines the tableau $T(\mathbf{i})$ wherein cells with \emph{content}\footnote{Recall that the content of a cell in row $a$ and column $b$ in the Young diagram of $\lambda$ is equal to $b-a$.} $k$ contain indices $j$ such that $i_j=m+k$. 

For each such $\mathbf{i}$,  that $w$ is Grassmannian forces $\psymb{\mathbf(i)}$ to be an indexed linear tree with support $[m-d,m+n-2-d]$ where $d$ is the number of indices $j$ such that $i_j>i_{j+1}$. But $d$ correspond precisely to the number of descents in $T(\mathbf{i})$ via the above bijection, so that $\mathbf{i}$ is $\Omega$-parking when $d=m-1$. 
Thus Theorem~\ref{th:the_theorem} gives our previous interpretation for $a_w$ for $w$ Grassmannian.

\item Let $w$ be a permutation in $S'_n$ such that the letters occurring in any element (equivalently, all elements) of $\red(w^{-1})$ form an initial interval $\{1,\ldots,p-1\}$ for some $p\geq 2$. 
Define $1^{j}\times w$ to be the permutation in $S_n$ given by $i+j\mapsto w(i)+j$ for $i=1,\ldots,p-1$ and having fixed points elsewhere. Note that $\red(1^{j}\times w)$ is obtained by applying $i\to i+j$ to letters in all words in $\red(w)$.  
Then \cite[Theorem 1.3.4]{NT20} states
\begin{equation}
\label{eq:nice_summation}
\sum_{j\geq 0}\schub{1^j\times w}(1,\ldots,1)t^j=\frac{{\displaystyle{\sum_{m=0}^{n-p-1}a_{1^m\times w}t^m}}}{(1-t)^n}.
\end{equation}
Let us now see how to retrieve this result by the approach in this paper.  We know \[\schub{u}=\sum_{P\in\red(u^{-1})/\wfequiv}\forestpoly(P)\] for any $u\in S_n$. Now let $u=1^j\times w$ as above. Then the LBS forests occurring in the sum are necessarily trees since, all words in $\red(u^{-1})$ have their letter values forming the interval $[j+1,j+p-1]$ in $\bZ$. 
So the supports of any such $P$ have the form $[1+m,n+m-1]$ for $m\in\bZ$; since this support contains $[j+1,j+p-1]$ we must have $p+j-n\leq m\leq j$. Furthermore $m\geq 0$ since $\forestpoly(P)=0$ otherwise.
\begin{align*}
\schub{u}&=\sum_{m=\min(0,p+j-n)}^{j}\sum_{\substack{P\in\red(u^{-1})/\wfequiv\\ \supp(P)=[1+m,n+m-1]}}\forestpoly(P)
\end{align*}
Note that the term for $m=0$  is precisely the reduction of $\schub{w}$, and we deduce that for any term for $m\geq 0$ the reduction of $\schub{1^m\times w}$ by a clear translation property of the $\wf$-parking procedure. Evaluating at $x_i=1$ for $i\geq 1$ gives 
\begin{align*}
\schub{u}(1,\ldots,1)&=\sum_{m=\min(0,p+j-n)}^{j}\sum_{\substack{P\in\red(u^{-1})/\wfequiv\\ \supp(P)=[1+m,n+m-1]}}\forestpoly(P)(1,\ldots,1)\\
&=\sum_{m=\min(0,p+j-n)}^{j}\sum_{\substack{P\in\red(u^{-1})/\wfequiv\\ \supp(P)=[1+m,n+m-1]}}|\decreasing(P)|\binom{n+m-1}{n-1}\\
&=\sum_{m=\min(0,p+j-n)}^{j}a_{1^{j-m}\times w}\binom{n+m-1}{n-1}.
\end{align*}
We used Proposition~\ref{prop:all_1s_tree} in the second equality, and we recognized our interpretation for $a_{1^{j-m}\times w}$ in the last one. The final  result is  equivalent to \eqref{eq:nice_summation}, as it encodes the equality of coefficients of $t^j$ on both sides.


\item There exists a cyclic sum rule for the numbers $a_w$, given as \cite[Theorem 1.3.4]{NT20}. Starting from our interpretation of $a_w$ as $\wf$-parking words, this rule can be proved by a variant of Pollak's argument for the usual parking functions. This will be presented in a more general context in forthcoming work of the first author~\cite{NadParking}. 
\item We also have the three striking properties for all  $w\in S'_{n}$(\cite[Theorem 1.3.1]{NT20}): $a_w>0$,  $a_w=a_{w^{-1}}$, and $a_w=a_{w_0ww_0}$ where $w_0$ is the longest element in $S_n$.

 The first one now translates into the fact that there exists a $\wf$-parking word in $\red(w^{-1})$. 
 Despite the simplicity of this statement we have not been able to explicitly find such a word in general. 
 The second statement suggests a certain symmetry property of the $\Omega$-parking procedure that eludes us. The third one follows from an elementary left-right symmetry of our $\wf$-parking algorithm and is left to the reader.
\end{enumerate}

\subsection{More about forest polynomials}
 We have introduced forest polynomials in this work in order to solve a particular problem of Schubert calculus. These polynomials turned out to possess various nice properties, starting with the fact that they form a basis with positive integral structure constants \ref{th:mult_forest_poly}. 
  In upcoming work, we will prove several more combinatorial and algebraic properties of these polynomials. For instance, they will help us determine the harmonic polynomials for the ideal $\qsym_n^+$. The $\wfcorr$-correspondence plays a key rule in determining these harmonics and so does the cubical subdivision of the permutahedron discussed in \cite{NTremixed}.
  
  We would be remiss not to mention the recent works \cite{BG23,PS23} which, informally speaking, discuss geometric aspects of quasisymmetric polynomials or $\qsym_n^+$. We are yet to explore the connection between these works and our results, but that is an avenue worthy of investigation.
  
 
\appendix

\section{Proof of Theorem~\ref{th:forest_into_dual_forest}}
\label{app:sign_reversing}

Given a lower ideal $L$ in $F$, let $U$ denote its complement. 
We define $\max(L)$ (respectively $\min(U)$) to be the set of maximal  (respectively minimal) elements in $L$ (respectively $U$).
Here is the key definition guiding our involution.
\begin{definition}\label{def:exchangeable}
	Given a lower ideal $L$ in $F$, fix $(\kappa,L)\in \compl(F)$.
	We say that $v\in \internal(F)$ is \emph{exchangeable}  if one of the following 	conditions hold:
\begin{enumerate}[align=parleft,left=0pt..1.5em,label=\text{(\arabic*)}]
		\item \label{it1:exchangeable}$v\in \max(L)$ and either $v$ has no parent or $\kappa(v) \geq \kappa(\mathrm{parent}(v))$ and $v$ is a left child, or  $\kappa(v)>\kappa(\mathrm{parent}(v))$ and $v$ is a right child; 
		\item \label{it2:exchangeable} $v\in \min(U)$, $\kappa(v)>\rho_F(v)$,  and $\kappa(v)$ is both strictly greater than the label of its left child (provided it is in $\internal(F)$) and weakly greater  than the label of its right child (provided it is in $\internal(F)$) respectively. 
	\end{enumerate}
\end{definition}
As the name suggests, an exchangeable node can be transferred from $L$ to $U$, or vice versa, while maintaining the weight but flipping the sign.
Before employing this aspect, we need to ensure that we can find such nodes.

Let $\mathsf{Good}\subset \compl(F)$ contain $(\kappa,\emptyset)$ such that $\kappa(v)\leq \rho_F(v)$ for all $v$. Note that 
\begin{align}
\label{eq:why_care_for_good?}
\forestpoly_F=\sum_{(\kappa,\emptyset)\in \mathsf{Good}} \alpx_{\kappa}=\sum_{y\in \mathsf{Good}}\wt(y).
\end{align}
\begin{lemma}\label{lem:good_does_not_exist}
Fix $(\kappa,L)\in \compl(F)$. There exists an exchangeable node if and only if $(\kappa,L)\notin \mathsf{Good}$.
\end{lemma}
\begin{proof}

One direction of the claim is immediate. If $(\kappa,L)\in \mathsf{Good}$, then no internal node satisfies the conditions in Definition~\ref{def:exchangeable}. 
So we focus on establishing the other direction.

Assume $(\kappa,L)\notin \mathsf{Good}$. 

Suppose first that each tree in $F$ is either entirely in $L$ or in $U$. If at least one tree is in $L$, its root is exchangeable by Definition~\ref{def:exchangeable}\ref{it1:exchangeable}. If all trees are in $U$, then $(\kappa,L)\notin \mathsf{Good}$ implies that there exists $v\in U$ such that $\kappa(v)>\rho_F(v)$, and $v$ can be assumed to be a terminal node. 
Then $v$ is exchangeable by Definition~\ref{def:exchangeable}\ref{it2:exchangeable}.

We can now assume that $F$ has a tree with vertices in both $L$ and $U$. Consider $v\in U$ with  at least one child in $L$, and minimal with respect to this condition. 
There are then three cases depending on which children of $v$ are in $L$.
\begin{enumerate}[align=parleft,left=0pt,label=\textbf{Case \Alph*}:]

\item Suppose that $u,w\in \max(L)$ are the left and right children of $v$ respectively.
In particular $v\in \min(U)$.
If $\kappa(u)\geq \kappa(v)$  then $u$ is exchangeable. 
If $\kappa(v)<\kappa(w)$, then $w$ is exchangeable. 
Otherwise we must have $\kappa(u)<\kappa(v)$ and $\kappa(v)\geq \kappa(w)$. 
The condition $\kappa(v)>\rho_F(v)$ is now implicit in the preceding inequalities as $\rho_F(v)=\rho_F(u)$, and $\kappa(u)>\rho_F(u)$ as $u\in L$.
Thus $v$ is exchangeable by Definition~\ref{def:exchangeable}\ref{it2:exchangeable}. 

\item 
Suppose $u\in\max(L)$ is the left child of $v$ and that the right child of $v$ is the node $w\in U$. Note that $v\notin \min(U)$ in this scenario, and that all children of $w$ are in $U$, by the minimality assumption on $v$.
If $\kappa(u)\geq \kappa(v)$, then $u$ is exchangeable by Definition~\ref{def:exchangeable}\ref{it1:exchangeable}. 

So suppose $\kappa(u)<\kappa(v)$.  Let $u'$ be the rightmost node reachable from $u$ by following right edges. Then we know that $\kappa(u)\geq \kappa(u')> \rho_F(u')$.
Now let $w'$ be the leftmost terminal node in the tree rooted at $w$ (in particular, $w'$ could potentially be $w$ itself). Clearly $w'\in \min(U)$.
We claim that $\kappa(w')>\rho_F(w')$, and thereby that $w'$ is exchangeable by Definition~\ref{def:exchangeable}\ref{it2:exchangeable}.
To see this, suppose first that there are $k\geq 1$ right edges on the unique path from $v$ to $w'$.
Then it must be the case that 
\begin{align}
	\rho_F(w')=\rho_F(u')+k+1.
\end{align}
Also since $v$ and $w'$ are both in $U$, given that we have strict increases down right edges, we must have
\begin{align}
	\kappa(w')\geq   \kappa(v)+k
\end{align}
Since $\kappa(v)>\kappa(u)>\rho_F(u')$, we thus infer that 
\begin{align}
	\kappa(w')> \rho_F(u')+1+k=\rho_F(w'),
\end{align}
which is what we wanted to show.

We now consider the case where $v$ has no right child (or equivalently that the right child of $v$ is a leaf). Thus we must have $v\in \min(U)$. If $\kappa(u)\geq \kappa(v)$, we again have that $u$ is exchangeable by Definition~\ref{def:exchangeable}\ref{it1:exchangeable}.  So we assume that $\kappa(u)<\kappa(v)$. 
Since $\rho_F(u)=\rho_F(v)$, it follows  that $v$ is exchangeable by Definition~\ref{def:exchangeable}\ref{it2:exchangeable}.

\item Now suppose $w\in \max(L)$ is the right child of $v$ and the left child of $v$ is the node $u\in U$. 
Note that $u\notin\min(U)$ in this scenario. The minimality assumption on $v$ implies that all descendants of $u$ are in $U$.
If $\kappa(v)<\kappa(w)$, then $w$ is exchangeable by Definition~\ref{def:exchangeable}\ref{it1:exchangeable}.
So assume that $\kappa(v)\geq \kappa(w)$.

Let $u'$ be the rightmost terminal node in the tree rooted at $u$ (again, $u'$ could very well be $u$). 
Clearly $u'\in \min(U)$. 
If there are $k\geq 0$ right edges in the path from $v$ to $u'$, then we know that 
\begin{align}
\kappa(u')\geq \kappa(v)+k.
\end{align}
Now we know that $\kappa(v)\geq \kappa(w)>\rho_F(w)$, implying that $\kappa(u')>\rho_F(w)+k>\rho_F(u')$.
Thus $u'$ is exchangeable by Definition~\ref{def:exchangeable}\ref{it2:exchangeable}. 

Suppose finally that left child of $v$ is not in $U$, i.e. it is a leaf.
Then $v\in \min(U)$. Again, if $\kappa(v)<\kappa(w)$, we are done by Definition~\ref{def:exchangeable}\ref{it1:exchangeable} as $w$ is exchangeable.
So suppose $\kappa(v)\geq \kappa(w)>\rho_F(w)$. 
It is then immediate that $\kappa(v)>\rho_F(v)$. 
Thus $v$ is exchangeable by Definition~\ref{def:exchangeable}\ref{it2:exchangeable}.
\end{enumerate}
This finishes the proof.
\end{proof}

Having established that exchangeable nodes exist in all bad cases, we are in position to prove Theorem~\ref{th:forest_into_dual_forest}.
\begin{proof}[Proof of Theorem~\ref{th:forest_into_dual_forest}]
	Fix an indexed forest $F$. By~\eqref{eq:why_care_for_good?}, we have to show
		\begin{align*}
		\sum_{y\in \compl(F)\setminus \mathsf{Good}} \wt(y)\mathrm{sign}(y)&=0.
	\end{align*}
	We do this via an involution on $\compl(F)$ that uses our notion of exchangeable nodes.
	
	We define a map $\Psi:\compl(F) \to \compl(F)$ as follows. Let $y=(\kappa,L)\in \compl(F)$.
	\begin{enumerate}
	\item If $y\in \mathsf{Good}$, then $\Psi(y)=y$.
	\item Otherwise let $v$ be the exchangeable node that occurs first in inorder.
	\begin{itemize}
	\item If $v\in \max(L)$, then define $\Psi(y)=(\kappa,L-\{v\})$. 
	\item If $v\in \min(U)$, then define $\Psi(y)= (\kappa,L+\{v\})$.
	\end{itemize}
	\end{enumerate}
	By Lemma~\ref{lem:good_does_not_exist}, we have that $\Psi$ is well defined.
	By its definition we see that $\Psi$ is weight-preserving on $\compl(F)$, and sign-reversing on $\compl(F)\setminus \mathsf{Good}$.
	
	All that remains to be verified is that $\Psi$ is indeed an involution.
	We leave it to the reader to check that if $v$ is the first exchangeable node in inorder for $y$, then it is also the first exchangeable node in inorder for $\Psi(y)$.
	In fact, if $v$ is an exchangeable node satisfying the criterion in Definition~\ref{def:exchangeable}\ref{it1:exchangeable} for $y$, then it becomes an exchangeable node satisfying the criterion in Definition~\ref{def:exchangeable}\ref{it2:exchangeable} for $\Psi(y)$, and vice versa.
\end{proof}

\bibliographystyle{alpha}
\bibliography{Biblio_mod_qsym}
\end{document}